\documentclass[12pt]{amsart}
\usepackage[margin=1.35in]{geometry}
\usepackage{amscd,amsmath,amsxtra,amsthm,amssymb,stmaryrd,xr,mathrsfs,mathtools,enumerate,commath, comment}
\usepackage{stmaryrd}
\usepackage{multirow}
\usepackage{xcolor}
\usepackage{commath}
\usepackage{comment}
\usepackage{tikz-cd}
\usepackage{longtable} 
\usepackage{pdflscape} 
\usepackage{booktabs}
\usepackage{hyperref}
\definecolor{vegasgold}{rgb}{0.77, 0.7, 0.35}
\definecolor{darkgoldenrod}{rgb}{0.72, 0.53, 0.04}
\definecolor{gold(metallic)}{rgb}{0.83, 0.69, 0.22}
\hypersetup{
 colorlinks=true,
 linkcolor=darkgoldenrod,
 filecolor=brown,      
 urlcolor=gold(metallic),
 citecolor=darkgoldenrod,
 }
\newtheorem{lthm}{Theorem}

\usepackage[all,cmtip]{xy}

\DeclareFontFamily{U}{wncy}{}
\DeclareFontShape{U}{wncy}{m}{n}{<->wncyr10}{}
\DeclareSymbolFont{mcy}{U}{wncy}{m}{n}
\DeclareMathSymbol{\Sh}{\mathord}{mcy}{"58}
\usepackage[T2A,T1]{fontenc}
\usepackage[OT2,T1]{fontenc}

\newtheorem{theorem}{Theorem}[section]
\newtheorem{lemma}[theorem]{Lemma}
\newtheorem{notn}[theorem]{Notation}

\newtheorem*{theorem*}{Theorem}
\newtheorem*{ass*}{Assumption}
\newtheorem{definition}[theorem]{Definition}
\newtheorem{corollary}[theorem]{Corollary}
\newtheorem{remark}[theorem]{Remark}

\newtheorem{proposition}[theorem]{Proposition}

\newcommand{\ord}{\mathrm{ord}}

\newcommand{\cN}{\mathcal{N}}

\newcommand{\Z}{\mathbb{Z}}
\newcommand{\cW}{\mathcal{W}}

\newcommand{\Q}{\mathbb{Q}}
\newcommand{\cG}{\mathcal{G}}

\newcommand{\F}{\mathbb{F}}

\newcommand{\cR}{\mathcal{R}}

\newcommand{\cS}{\mathcal{S}}

\newcommand{\V}{\mathcal{V}}

\newcommand{\op}[1]{\operatorname{#1}}

\numberwithin{equation}{section}

\begin{document}

\title[Counting rational maps]{Counting rational maps on $\mathbb{P}^1$ with prescribed local conditions}

\author[K.~D.~Nguyen]{Khoa D.~Nguyen}
\address[Nguyen]{Department of Mathematics and Statistics, University of Calgary,
2500 University Drive NW, Calgary, T2N 1N4, Alberta, Canada}
\email{dangkhoa.nguyen@ucalgary.ca}

\author[A.~Ray]{Anwesh Ray}
\address[Ray]{Chennai Mathematical Institute, H1, SIPCOT IT Park, Kelambakkam, Siruseri, Tamil Nadu 603103, India}
\email{anwesh@cmi.ac.in}


\begin{abstract}
We explore distribution questions for rational maps on the projective line $\mathbb{P}^1$ over $\mathbb{Q}$ within the framework of arithmetic dynamics, drawing analogies to elliptic curves. Specifically, we investigate counting problems for rational maps $\phi$ of fixed degree $d \geq 2$ with prescribed reduction properties. One of our main results establishes that with respect to the weak box densities in an earlier work of Poonen, a positive proportion of rational maps consist of those having globally minimal resultant. Additionally, for degree 2 rational maps, we perform explicit computations demonstrating that over $32.7\%$ possess a squarefree, and hence minimal, resultant.
\end{abstract}

\maketitle

\section{Introduction}

\subsection{Background and motivation} 
\par An important class of algebraic  dynamical systems arises from rational functions on the projective line $\mathbb{P}^1$ with coefficients in the field of rational numbers $\Q$. These systems are studied in the burgeoning field of arithmetic dynamics, which seeks to understand the arithmetic properties of the iterates of such maps. Rational functions $\phi$ in this context are viewed as global objects, and their arithmetic dynamics share many analogous properties to the diophantine geometry of elliptic curves defined over $\Q$.  
A point $\alpha$ is called \emph{preperiodic} if its forward orbit is finite, and a \emph{wandering point} otherwise. These are analogous to torsion points and points of infinite order on elliptic curves. 

For an elliptic curve defined over a number field, one can define the local minimal discriminant at each non-archimedean place and define the (global) minimal discriminant as the product of the local minimal discriminants \cite[Chapters~VII--VIII]{silvermanEC}. Similarly, for a rational map defined over a number field, one can define the local minimal resultant at each non-archimedean place and define the (global) minimal resultant as their product, see \cite[Chapter~4.11]{silvermantextbook}. After homogenization, a rational map $\phi\in\Q(z)$  is given by a pair of homogeneous 
polynomials $F(X,Y),G(X,Y)\in \Z[X,Y]$ without common factors in $\Z[X,Y]$. We say that $\phi$ is globally minimal or has globally minimal resultant if the resultant of $F$ and $G$ is equal to the minimal resultant of $\phi$. Over a complete algebraically closed non-archimedean valued field, the notion of (local) minimal resultants give rise to the so called minimal resultant locus which has interesting analytic and moduli-theoretic characterizations \cite{Rumely2015,Rumely2017}.  Over $\Q$, globally minimal maps play an important role in a conjecture of Silverman \cite[Conjecture 3.47]{silvermantextbook} on the uniformity of number of integral points in orbits. This is, in turn, analogous to a conjecture of Lang \cite{Lang} for integral points on elliptic curves after work of Dem'janenko.

In recent years, there have been various results concerning  statistical aspects of the arithmetic dynamics on $\mathbb{P}^1$, for example \cite{HH2019,H2021a,H2021b,DHHJ2024}. However, the nature of these results involves a stochastic process in which one successively composes a random rational map from a finite set endowed with a probability measure. On the other hand, our work continues the classical analogy program between arithmetic dynamics on $\mathbb{P}^1$ and elliptic curves. 
In arithmetic statistics of elliptic curves, considerable interest lies in questions regarding the distribution of sets $\cS$ with certain properties. For instance, the rank distribution conjecture of Katz and Sarnak predicts that exactly $1/2$ (resp. $1/2$) of the elliptic curves over $\Q$ have rank $0$ (resp. $1$). Cremona and Sadek \cite{cremonasadek} consider the density of global Weierstrass models of elliptic curves over $\Q$ with prescribed reduction types at the primes. We prove analogues of results of Cremona and Sadek for models for dynamical systems, arising from rational maps of prescribed degree $d\geq 2$ in $\mathbb{P}^1$ with coefficients in $\Q$.

\subsection{Main results}
\par Heights of polynomials are standard in diophantine geometry and we use them to formulate counting problems related to arithmetic dynamics. Fix $d\geq 2$, let $\cR^{(d)}$ denote the
set of univariate rational maps of degree $d$ with rational coefficients.
We count $\phi(z)\in\cR^{(d)}$ with certain prescribed reduction types. By homogenization, we write $\phi(X/Y) = \frac{F(X, Y)}{G(X, Y)}$ where $F$ and $G$ are coprime homogeneous polynomials given by:
\[ 
F(X, Y) = \sum_{i=0}^d A_i X^i Y^{d-i} \quad \text{and} \quad G(X, Y) = \sum_{j=0}^d B_j X^j Y^{d-j}, 
\]
with $A_i, B_j \in \Q$. To the map $\phi$, one associates the point
\[ 
P(\phi) := [A_0 : \cdots : A_d : B_0 : \cdots : B_d] \in \mathbb{P}^{2d+1}(\Q). 
\]
The height of $\phi$, denoted $H(\phi)$, is defined to be the absolute multiplicative Weil height 
of the point $P(\phi)$ \cite[Chapter~1]{BomGub}.   
For any $x > 0$, let $\cR^{(d)}(x)=\{\phi\in\cR^{(d)}:\ H(\phi)\leq x\}$. It is seen that  \[\# \cR^{(d)}(x)= \frac{(2x)^{2d+2}}{2\zeta(2d+2)}+O\left( x^{2d+1}\log x\right),\] cf. Proposition \ref{count for Rd}. Given a set of rational maps $\cS\subseteq \cR^{(d)}$, we define its height density to be
\[\lim_{x\rightarrow\infty} \left(\frac{\#\{\phi\in \cR^{(d)}(x)\mid \phi\in \cS\}}{\# \cR^{(d)}(x)}\right),\]
provided that the limit exists. We say that with respect to the height density more than $c \%$ or rational maps satisfy $\cS$ if 
\[\liminf_{x\rightarrow\infty} \left(\frac{\#\{\phi\in \cR^{(d)}(x)\mid \phi\in \cS\}}{\# \cR^{(d)}(x)}\right)> \frac{c}{100}.\]

Inspired by Poonen's work \cite{Poonensquarefree}, there is a more general density to consider. First, it is possible to regard $\cR^{(d)}\subset \mathbb{P}^1(\Q)$ as subsets of $\Z^{2d+2}$, see Notation~\ref{notation}. Then one may consider the \emph{box density}  in 
\cite[p.~355]{Poonensquarefree}. There are weaker versions of this called the \emph{weak box densities}, cf. Theorem 3.2 of \emph{loc. cit}. More specifically, for $\mathcal{S}\subset \cR^{(d)}$ and for each
natural number $n\in [1,2d+2]$, we have the notion of the weak box density
$\mu_n(\mathcal{S})$, see Definition~\ref{def:weak box densities}.

In this paper, we obtain results for the less general densities, namely the height density and the weak box densities. Our results for the height density hold for \emph{odd} values of $d$ while results for the weak box densities hold for every $d\geq 2$:

\begin{lthm}[Remark~\ref{rem:V and W have globally minimal maps} and Theorem~\ref{main thm of section 4 Wsigma}]\label{thm a1}
    Let $d\geq 3$ be an odd integer and let $\Sigma$ be a finite set of prime numbers. Let $\cW_\Sigma$ be the set of rational maps $\phi(z)\in\Q(z)$ of degree $d$ with $2d$-power free resultant and good reduction at all primes $p\in \Sigma$. The following assertions hold:
    \begin{enumerate}
        \item each rational map $\phi\in \cW_\Sigma$ has globally minimal resultant;
        \item $\cW_\Sigma$ has positive height density.
    \end{enumerate}
    Consequently, with respect to the height density, a positive proportion of rational maps have good reduction at all primes $p\in \Sigma$ and globally minimal resultant.
\end{lthm}

\begin{lthm}[Remark~\ref{rem:V and W have globally minimal maps} and Theorem~\ref{main thm of section 4}]\label{thm a2}
    Let $d\geq 2$ and $n\in [1, 2d+2]$ be integers and let $\Sigma$ be a finite set of prime numbers. Let $\mathcal{V}_\Sigma$ be the set of rational maps $\phi(z)\in\Q(z)$ of degree $d$ with $d$-power free resultant and good reduction at all primes $p\in \Sigma$. The following assertions hold: 
    \begin{enumerate}
        \item each rational map $\phi\in \mathcal{V}_\Sigma$ has minimal resultant;
        \item $\mathcal{V}_\Sigma$ has positive $\mu_n$ density.
    \end{enumerate}Consequently, with respect to the weak box density $\mu_n$, a positive proportion of rational maps have good reduction at all primes $p\in \Sigma$ and globally minimal resultant.
\end{lthm}

In the special case when $d=2$, our results are quite explicit.
\begin{lthm}[Theorem \ref{d=2 main thm}]\label{thm b}
    With respect to the weak box densities, more than $32.7\%$ of rational maps of degree $2$ over $\Q$ have squarefree, hence globally minimal, resultant.
\end{lthm}

The following analogue of the result of Cremona-Sadek 
\cite{cremonasadek} arises: with respect to the box density, does a positive proportion of rational maps have globally minimal resultant? This question, along with related topics, will be explored in a future work.

\subsection{Methodology}  We briefly discuss the techniques used in the proof of Theorems \ref{thm a1}, \ref{thm a2} and \ref{thm b}. A classical result of Schanuel \cite{schanuel} gives the asymptotic count for the number of rational points in projective space with height $\leq x$. For two generic homogeneous polynomials $F(X,Y)$ and $G(X,Y)$ of degree $d$, 
their resultant is a homogeneous polynomial in the coefficients of $F$ and $G$. 
Hence $F$ and $G$ have a common factor if and only if the resultant polynomial vanishes. A result by Schwartz and Zippel (cf. Lemma \ref{Schzipp lemma}) provides an upper bound on the number of zeros of this polynomial with height $\leq x$. Combining this result with Schanuel's result implies that
\[
\# \cR^{(d)}(x) \sim \frac{(2x)^{2d+2}}{2\zeta(2d+2)},
\]
as detailed in Proposition \ref{count for Rd}. Following Poonen's work \cite{Poonensquarefree}, we can also study the ``box version'' of the above results.

In Section \ref{s 3}, we study properties of subsets $\mathcal{S} \subseteq \cR^{(d)}$ defined by the so called local conditions. We define local densities at each prime $p$. Proposition \ref{prop:symbol delta admissible} employs a sieve-theoretic argument to show that, if a certain admissibility condition is met, the density of $\mathcal{S}$ is the product of local densities at all primes $p$. These results are applied in Section \ref{s 4} to prove Theorems \ref{thm a1} and \ref{thm a2}. Although these local conditions are insufficient to describe the entire collection of $\phi$ with minimal resultant, they provide positive lower bounds on the density. This relies on explicit local density calculations, with a main innovation involving counting the number of $\mathbb{Z}/p^2\mathbb{Z}$-points on the integral scheme defined by the resultant polynomial and its partial derivatives. Some similar arguments have been employed in \cite{Poonensquarefree}.
In Section \ref{s 5}, an explicit primary decomposition for this integral scheme is used to obtain an explicit lower bound on the weak box densities when $d=2$. Our computations are aided by  \textit{Macaulay2}.

\subsection{Acknowledgements} The authors thank Igor Shparlinski for helpful comments. K.~N.~ is partially supported by NSERC grant RGPIN-2018-03770 and CRC tier-2 research stipend 950-231716 from the Government of Canada.

\subsection*{Data availability statement} All relevant data has been included in the paper.

\section{The height density, box density, and weak box densities}
Let $d\geq 2$ and let $\cR^{(d)}$ be the set of univariate rational maps of degree $d$ with rational coefficients. This section establishes key concepts and notation for this article regarding the distribution of certain subsets of $\cR^{(d)}$. 
We may regard $\cR^{(d)}$ as a subset of $\mathbb{P}^{2d+1}(\Q)$ which, in turn, might be regarded as a subset of $\Z^{2d+2}$.

We define the height $H(\phi)$ of $\phi\in\cR^{(d)}$ as the absolute multiplicative Weil height of the corresponding point in $\mathbb{P}^{2d+1}(\Q)$. 
Then we define the height density of a subset $\mathcal{S}$ of $\cR^{(d)}$. 
Using classical results of Schanuel \cite{schanuel} for the distribution of rational points on projective spaces and the Schwartz-Zippel lemma \cite{Schwartz,Zippel}, we derive an asymptotic formula for 
$$\cR^{(d)}(x):=\{\phi\in\cR^{(d)}:\ H(\phi)\leq x\}.$$

We also recall the notion of box density and weak box densities for subsets of 
$\Z^{2d+2}$ 
from Poonen's work \cite{Poonensquarefree}. We have the analogues of
Schanuel's result and an asymptotic formula for counting $\cR^{(d)}$ in a box; however the statement of these analogues is delayed until the next section since the proof relies on certain sieving results. Finally we define
the box density and weak box densities for a subset $\mathcal{S}$ of $\cR^{(d)}$.
 
\subsection{The height density}
Let $\phi\in\cR^{(d)}$. In homogeneous coordinates, $\phi$ is given by $\phi(X/Y) = \frac{F(X, Y)}{G(X, Y)}$, where $F$ and $G$ are coprime homogeneous polynomials of degree $d$:
\[ F(X, Y) = \sum_{i=0}^d A_i X^i Y^{d-i} \quad \text{and} \quad G(X, Y) = \sum_{j=0}^d B_j X^j Y^{d-j}, \]
with $A_i, B_j \in \mathbb{Q}$.
\begin{definition}
    To the map $\phi$, we associate the point 
    \[ P=P(\phi) := [A_0 : \cdots : A_d : B_0 : \cdots : B_d] \in \mathbb{P}^{2d+1}(\mathbb{Q}). \]
     The \emph{height} $H(\phi)$ of the pair $\phi$ is defined as follows
    \begin{equation}\label{height of phi}
    H(\phi)=H(P):= \prod_v \max\{|A_i|_v, |B_j|_v:\ 0 \leq i, j \leq d\}, 
    \end{equation}
    where $v$ ranges over all places of $\mathbb{Q}$.
\end{definition}

Thus we count models for rational maps, and not isomorphism classes. Nonetheless, this seems to be natural way to order rational maps, the definition is analogous to counting Weierstrass equations for elliptic curves according to their height (cf. \cite{cremonasadek}). For a detailed treatment of heights, we refer to \cite{BomGub}.

\begin{notn}\label{notation}
From now on, we shall identify $\mathbb{P}^{2d+1}(\Q)$ with the subset of $\Z^{2d+2}$ consisting of points $Q=(C_1, \dots, C_{2d+2})$ for which 
\begin{enumerate}
    \item $C_i>0$, where $i$ is the minimal index for which $C_i\neq 0$. In other words, the first non-zero coordinate is positive.
    \item and $\gcd(C_1, \dots, C_{2d+2}) = 1$.
\end{enumerate}
We shall also denote the coordinates of $Q$ by $(A_0, \dots, A_d, B_0, \dots, B_d)$, i.e., $A_i:=C_{i+1}$ and $B_j:=C_{d+j+2}$. 
Note that if $Q$ satisfies (2), then exactly one of the points $Q$ or $-Q$ satisfies (1).
\end{notn}
Note that if $(A_0, \dots, A_d, B_0, \dots, B_d) \in \mathbb{Z}^{2d+2}$ corresponds to $ P(\phi) $, then
\[ H(\phi) = \max\{|A_i|, |B_j|: 0 \leq i, j \leq d\}, \] see \cite[Chapter~1]{BomGub}.

\begin{notn}\label{notation (x)}
Let $x>0$. For a subset $S$ of $\Z^{2d+2}$, let
$$S(x)=\{(C_1,\ldots,C_{2d+2})\in S:\ \max|C_i|\leq x\}.$$
For a subset $\mathcal{S}$ of $\mathbb{P}^{2d+1}(\Q)$, let
$$\mathcal{S}(x)=\{P\in\mathcal{S}:\ H(P)\leq x\}.$$
However we typically use the notation $\mathbb{P}^{2d+1}(\Q,x)$ instead
of $\mathbb{P}^{2d+1}(\Q)(x)$.
\end{notn}

We have the following well-known results by Schanuel, Schwartz, and Zippel:
\begin{proposition}[Schanuel]\label{schanuel prop}
    \begin{equation}\label{schanuel eqn}\# \mathbb{P}^{2d+1}(\Q, x)= \frac{1}{2\zeta(2d+2)}(2x)^{2d+2}+O\left( x^{2d+1}\log x\right).\end{equation}
\end{proposition}
\begin{proof}
    This is from \cite[Theorem 1]{schanuel}.
\end{proof}

\begin{lemma}[Schwartz-Zippel]\label{Schzipp lemma}
    Let $n$ be a positive integer, $R$ an integral domain, and $f\in R[x_1, \dots, x_n]$ a non-zero polynomial. There exists a positive constant $C$ depending only on $f$ such that the following holds. Let $I_1,\ldots,I_n$ be non-empty finite subsets of $R$, then the set
    $$\{(x_1,\ldots,x_n)\in I_1\times\cdots\times I_n:\ f(x_1,\ldots,x_n)=0\}$$
    has cardinality at most
    $$C\cdot |I_1|\cdots |I_n|\left(\frac{1}{|I_1|}+\cdots+\frac{1}{|I_n|}\right).$$
\end{lemma}
\begin{proof}
    This was proven independently by Schwartz \cite{Schwartz} and Zippel \cite{Zippel}.
\end{proof}

\begin{proposition}\label{count for Rd}
   We have 
    \[\# \cR^{(d)}(x)= \frac{1}{2\zeta(2d+2)}(2x)^{2d+2}+O\left( x^{2d+1}\log x\right),\] where the implied constant depends only on $d$.
    Consequently,
    \[\lim_{x\rightarrow\infty} \frac{\# \cR^{(d)}(x)}{\# \mathbb{P}^{2d+1}(\Q, x)}=1.\]
\end{proposition}
\begin{proof}
    Let $f(x_0,\ldots,x_d,y_0,\ldots,y_d)$ be the resultant of the two generic polynomials $\displaystyle \sum_{i=0}^d x_i X^i Y^{d-i} $ and
    $\displaystyle \sum_{i=0}^d y_i X^i Y^{d-i}$. 
    A point $(A_0, \dots, B_d)$ in $\mathbb{P}^{2d+1}(\Q)$ (see Notation~\ref{notation}) is not in $\cR^{(d)}$ if and only if
    the polynomials $ F(X, Y) = \sum_{i=0}^d A_i X^i Y^{d-i} $ and $ G(X, Y) = \sum_{j=0}^d B_j X^j Y^{d-j} $ have a non-trivial common factor, hence
    $f(A_0,\ldots,B_d)=0$. From Lemma~\ref{Schzipp lemma}, we have
    $$\#(\mathbb{P}^{2d+1}(\Q,x)\setminus \cR^{(d)}(x))=O(x^{2d+1})$$
    where the implied constant depends only on $f$, hence only on $d$. Then we apply
    Proposition~\ref{schanuel prop} to finish the proof.
\end{proof}

\begin{definition}Let $\cS$ be a subset of $\mathbb{P}^{2d+1}(\Q)$, the \emph{height density} of $\cS$ is 
\[\mathfrak{d}(\cS):=\lim_{x\rightarrow \infty} \frac{\# \cS(x)}{\# \mathbb{P}^{2d+1}(\Q,x)},\] provided the limit exists. The lower and upper height densities are 
\[\underline{\mathfrak{d}}(\cS):=\liminf_{x\rightarrow \infty} \frac{\# \cS(x)}{\# \mathbb{P}^{2d+1}(\Q,x)}\text{ and }\overline{\mathfrak{d}}(\cS):=\limsup_{x\rightarrow \infty} \frac{\# \cS(x)}{\# \mathbb{P}^{2d+1}(\Q,x)}.\]
Let $S$ be a subset of $\Z^{2d+2}$, the affine height density of $S$ is
$$\mathfrak{d}^*(S)=\lim_{x\rightarrow \infty} \frac{\# S(x)}{(2x)^{2d+2}}$$
provided the limit exists. The lower and upper affine height densities $\underline{\mathfrak{d}}^*$ and $\overline{\mathfrak{d}}^*$
are defined using $\liminf$ and $\limsup$ respectively.
\end{definition}

\begin{remark}\label{rem:height density}
What we call the affine height density is simply called the natural density in the literature. The reason for our somewhat unusual terminology and notation will become clear in the next section when we have various subsets of projective spaces and their affine liftings. When $\mathcal{S}$ is a set of rational maps of degree $d$, it appears more natural to replace $\#\mathbb{P}^{2d+1}(\Q,x)$
by $\#\cR^{(d)}(x)$ in the definition of $\mathfrak{d}(S)$, $\underline{\mathfrak{d}}(S)$, and $\overline{\mathfrak{d}}(S)$. But this is equivalent to the given definition thanks to Proposition~\ref{count for Rd}.
\end{remark}

\subsection{The box density and weak box densities}

Given a tuple of positive real numbers $\vec{r}=(r_1, \dots, r_{2d+2})$, let 
\[\op{Box}(\vec{r}):=\{(x_1, \dots, x_{2d+2})\in\Z^{2d+2}:\ |x_i|\leq r_i\text{ for all }i\in [1, 2d+2]\}.\]
\begin{notn}\label{notation (x) for proj}
Let $x>0$. For a subset $S$ of $\Z^{2d+2}$, let
$$S(\vec{r}):=\{(C_1,\ldots,C_{2d+2})\in S:\ |C_i|\leq r_i\ \text{for $1\leq i\leq 2d+2$}\}.$$
For a subset $\mathcal{S}$ of $\mathbb{P}^{2d+1}(\Q)$, let
$$\mathcal{S}(\vec{r}):=\{(C_1,\ldots,C_{2d+2})\in\mathcal{S}\ \text{as in Notation~\ref{notation}}:\ |C_i|\leq r_i\ \text{for $1\leq i\leq 2d+2$}\}.$$
However we typically use the notation $\mathbb{P}^{2d+1}(\Q,\vec{r})$ instead of
$\mathbb{P}^{2d+1}(\Q)(\vec{r})$.
\end{notn}

\begin{definition}
    Let $\mathcal{S}$ be a subset of $\mathbb{P}^{2d+1}(\Q)$, the box density of $\mathcal{S}$ is
    $$\mu(\mathcal{S}):=\lim_{r_1,\dots,r_{2d+2}\rightarrow \infty}\frac{\# \mathcal{S}(\vec{r})}{\#\mathbb{P}^{2d+2}(\Q,\vec{r})}$$
    provided the limit exists. The lower and upper box densities $\underline{\mu}(\mathcal{S})$ and $\overline{\mu}(\mathcal{S})$ are defined using $\liminf$ and $\limsup$ respectively.

    Let $S$ be a subset of $\Z^{2d+2}$, the affine box density of $S$ is
    $$\mu^*(S):=\lim_{r_1,\dots,r_{2d+2}\rightarrow\infty}\frac{\# S(\vec{r})}{\#\op{Box}(\vec{r})}$$
    provided the limit exists. The lower and upper affine box densities
    $\underline{\mu}^*(S)$ and $\overline{\mu}^*(S)$ are defined using 
    $\liminf$ and $\limsup$ respectively.
\end{definition}

Fix a natural number $n$ in the range $[1, 2d+2]$. In what follows, set 
\[\lim_{r_1, \dots, \hat{r}_{n}, \dots, r_{2d+2}\rightarrow \infty}:=\lim_{r_1, \dots, r_{n-1}, r_{n+1}, \dots, r_{2d+2}\rightarrow \infty},\] with $r_n$ omitted. The upper and lower limits 
\[\limsup_{r_1, \dots, \hat{r}_{n}, \dots, r_{2d+2}\rightarrow \infty}\text{ and } \liminf_{r_1, \dots, \hat{r}_{n}, \dots, r_{2d+2}\rightarrow \infty}\] are defined in a similar way.
\begin{definition}\label{def:weak box densities}
    Let $\mathcal{S}$ be a subset of $\mathbb{P}^{2d+1}(\Q)$. The 
    lower $n$-th weak box density $\underline{\mu_n}(\mathcal{S})$ and upper $n$-th weak box density $\overline{\mu_n}(\mathcal{S})$ are defined as:
    \[\begin{split} & \underline{\mu_n}(\cS):=\liminf_{r_1, \dots, \hat{r}_{n}, \dots, r_{2d+2}\rightarrow \infty}\left\{\liminf_{r_{n}\rightarrow \infty} \left(\frac{\# \cS(\vec{r})}{\# \mathbb{P}^{2d+1}(\Q,\vec{r})}\right)\right\}\ \text{and} \\
& \overline{\mu_n}(\cS):=\limsup_{r_1, \dots, \hat{r}_{n}, \dots, r_{2d+2}\rightarrow \infty}\left\{\limsup_{r_{n}\rightarrow \infty} \left(\frac{\# \cS(\vec{r})}{\# \mathbb{P}^{2d+1}(\Q,\vec{r})}\right)\right\}.
\end{split}\]
    When $\underline{\mu_n}(\mathcal{S})=\overline{\mu_n}(\mathcal{S})$,
    we define the $n$-th weak box density $\mu_n(\mathcal{S})$ to be the common value.

    Let $S$ be a subset of $\Z^{2d+2}$, we define the lower affine $n$-th weak box density $\underline{\mu_n}^*(S)$, the upper affine $n$-th weak box density $\overline{\mu_n}^*(S)$, and the affine $n$-th weak box density $\mu_n^*(S)$ in a similar manner by replacing
    $\#\mathbb{P}^{2d+1}(\Q,\vec{r})$ by $\# \op{Box}(\vec{r})$.
\end{definition}

We note that the definition of the affine densities in this subsection are taken from Poonen's paper \cite{Poonensquarefree}; however his terminologies do not include the word ``affine'' and his notations do not include the star superscript. Since $\cR^{(d)}$ is naturally a subset of $\mathbb{P}^{2d+1}(\Q)$, we reserve the simpler terminologies and notations for the projective setting.  As in Remark~\ref{rem:height density}, when $\mathcal{S}\subseteq \cR^{(d)}$, it is natural to define the various densities of $\mathcal{S}$ using $\# \cR^{(d)}(\vec{r})$ instead of
$\#\mathbb{P}^{2d+1}(\Q,\vec{r})$. This is equivalent to the given definition thanks to the box version of Proposition~\ref{schanuel prop}
and Proposition~\ref{count for Rd} established in the next section.
For $\mathcal{S}\subseteq\mathbb{P}^{2d+1}(\Q)$ (respectively
$S\subseteq \Z^{2d+2}$), if the box density $\mu(\mathcal{S})$
(respectively $\mu^*(S)$) exists, then the height density
$\mathfrak{d}(\mathcal{S})$ (respectively $\mathfrak{d}^*(S)$) and the
weak box density $\mu_n(\mathcal{S})$ (respectively $\mu_n^*(S)$) exist
and have the same value.

\section{Projective sets with prescribed local conditions and their affine liftings}\label{s 3}

For a positive integer $m$ and a commutative unital ring $R$, the projective space $\mathbb{P}^m(R)$ consists of the equivalence classes of
$$\{(x_0,\dots,x_m)\in R^{m+1}:\ \text{$x_i$ is a unit for some $i$}\}$$
under the relation $(x_0,\dots,x_m)\sim (y_0,\dots,y_m)$ iff 
$(x_0,\dots,x_m)=c(y_0,\dots,y_m)$ for some unit $c$. For $Q=(x_0,\dots,x_m)$ in the above set, the equivalence class of $Q$ in 
$\mathbb{P}^m(R)$ is denoted $[Q]$.

Throughout this section, fix an integer $d\geq 2$. We discuss subsets of
$\mathbb{P}^{2d+1}(\Q)$ defined by local conditions, their so called affine liftings, and some results concerning the various densities in the previous section.
We adopt the following heuristic to help the reader keep track of the notation. When using capital letters to denote a set of points, those in calligraphic font usually denote subsets of projective spaces, those in normal font usually denote subsets of affine spaces, and the bar notation denotes sets of points defined over the ring of integers modulo a prime power.

\subsection{Local conditions and affine liftings} 
\par  Given a point \[Q=(A_0, \dots, A_d, B_0, \dots, B_d)\in \mathbb{P}^{2d+1}(\Q)\] as in Notation~\ref{notation}, and a natural number $N\geq 2$, we let $Q_N$ be the reduction of $Q$ modulo $N$. The reduction is an element in $\mathbb{P}^{2d+1}(\Z/N\Z)$. 

Let $p$ be a prime. A subset $\mathcal{S}_p$ of $\mathbb{P}^{2d+1}(\Q)$ is defined by a local condition at $p$ if it consists of all points $P$ that reduce to a prescribed set of residue classes modulo a fixed power of $p$. In greater detail, there is an integer $n_p\geq 1$ and a set $\bar{\mathcal{S}}_p\subseteq \mathbb{P}^{2d+1}\left(\Z/p^{n_p} \Z\right)$ such that \[\mathcal{S}_p=\{P\in \mathbb{P}^{2d+1}(\Q)\mid P_{p^{n_p}}\in \bar{\mathcal{S}}_p\}.\]
We define 
\begin{equation}\label{eq:picSp definition}
\pi(\mathcal{S}_p):=\frac{\# \bar{\mathcal{S}}_p}{\# \mathbb{P}^{2d+1}\left(\Z/p^{n_p}\Z \right)}=\frac{p^{n_p-1}(p-1)\# \bar{\mathcal{S}}_p}{p^{(2d+2)n_p}\left(1-p^{-(2d+2)}\right)}.
\end{equation}
A set $\cS\subseteq \mathbb{P}^{2d+1}(\Q)$ is defined by local conditions if it is an intersection of the form $\cS=\bigcap_p \cS_p$, where $p$ ranges over all primes, and $\cS_p$ is defined by a local  condition at $p$. At each prime $p$, we set $\cS_p'=\mathbb{P}^{2d+1}(\Q)\setminus \cS_p$; note that $\pi(\cS_p')$ is also defined by a local condition at $p$ (since it is the preimage of 
the complement of $\bar{\mathcal{S}}_p$) and $\pi(\cS_p')=1-\pi(\cS_p)$.

Suppose $\mathcal{S}_p$ is defined by a local condition at $p$. The affine lifting of $\bar{\mathcal{S}}_p$ is the set:
\[\begin{split}\bar{S}_p:=& \{Q=(\bar{A}_0, \dots, \bar{A}_d, \bar{B}_0, \dots, \bar{B}_d)\in (\Z/p^{n_p}\Z)^{2d+2} \\ 
 &  \mid p\text{ does not divide some entry of }Q \text{ and }[Q]\in \bar{\mathcal{S}}_p\}.\end{split}\]
The affine lifting $S_p\subseteq \Z^{2d+2}$ of $\mathcal{S}_p$ is the preimage of $\bar{S}_p$ with respect to the reduction map $\Z^{2d+2}\rightarrow (\Z/p^{n_p}\Z)^{2d+2}$. We define
\begin{equation}\label{eq:piSp definition}
\pi(S_p):=\frac{\# \bar{S}_p}{p^{(2d+2)n_p}},
\end{equation}
then we have
\begin{equation}\label{S_p eqn}
\mu^*(S_p)=\lim_{r_1,\dots,r_{2d+2}\rightarrow \infty}\frac{\# S_p(\vec{r})}{\#\op{Box}(\vec{r})}=\pi(S_p).
\end{equation}
From $\# \bar{S}_p=\#(\Z/p^{n_p}\Z)^*\cdot\# \bar{\mathcal{S}}_p= p^{n_p-1}(p-1)\#\bar{\mathcal{S}}_p$, we have
 \begin{equation}\label{eq:pi vs pi}\pi(\mathcal{S}_p)=\frac{\#\bar{S}_p}{p^{(2d+2)n_p}\left(1-p^{-(2d+2)}\right)}=\frac{\pi(S_p)}{1-p^{-(2d+2)}}.\end{equation}

Finally, if $\mathcal{S}=\bigcap_p \cS_p$ is defined by local conditions then the 
affine lifting $S$ of $\mathcal{S}$ is the set
$S=\bigcap_p S_p$ where each $S_p$ is the affine lifting of
$\mathcal{S}_p$. It is immediate from the definition that $S$ consists precisely the points $Q=(C_1,\ldots,C_{2d+2})$ such that
$\gcd(C_1,\ldots,C_{2d+2})=1$ and $[Q]\in \mathcal{S}_p$ for every $p$. 
In other words, from Notation~\ref{notation}, $S$ consists 
of exactly pairs of points $Q$ and $-Q$ such that one of them belongs to $\mathcal{S}$. Hence for every $\vec{r}$, we have:
\begin{equation}\label{eq:S and calS}
    \# \mathcal{S}(\vec{r})=\frac{1}{2} \# S(\vec{r}). 
\end{equation}

We conclude this subsection with the box version of Proposition~\ref{schanuel prop} and Proposition~\ref{count for Rd}. The idea used in the proof will be extended in the next subsection.

\begin{lemma}\label{lem:T_>Z}
    Given a real number $Z$, let $T_{>Z}$ be the set of $(C_1,\ldots,C_{2d+2})\in\Z^{2d+2}$ such that $\gcd(C_1,\ldots,C_{2d+2})$ is divisible by a prime $p>Z$. We have
    $$\lim_{Z\rightarrow\infty}\limsup_{r_1,\ldots,r_{2d+2}\rightarrow\infty}\frac{\# T_{>Z}(\vec{r})}{\# \op{Box}(\vec{r})}=0.$$
\end{lemma}
\begin{proof}
    When $Z$ is fixed and $r_1,\ldots,r_{2d+2}\rightarrow\infty$, we have
    $$\# T_{>Z}(\vec{r})=O\left(\sum_{p>Z}\frac{r_1\cdots r_{2d+2}}{p^{2d+2}}\right)=O\left(\frac{r_1\cdots r_{2d+2}}{Z^{2d+1}}\right)$$
    where the implied constants depend only on $d$. This finishes the proof.
\end{proof}

\begin{proposition}\label{prop:box Schanuel and count Rd}
    \begin{align*}
    \lim_{r_1,\ldots,r_{2d+2}\rightarrow\infty}\frac{\#\mathbb{P}^{2d+1}(\Q,\vec{r})}{\#\op{Box}(\vec{r})}&=\frac{1}{2\zeta(2d+2)}.\\
    \lim_{r_1,\ldots,r_{2d+2}\rightarrow\infty}\frac{\#\cR^{(d)}(\vec{r})}{\#\mathbb{P}^{2d+1}(\Q,\vec{r})}&=1.
    \end{align*}
\end{proposition}
\begin{proof}
    As in the proof of Proposition~\ref{count for Rd}, the second identity follows from the first and Lemma~\ref{Schzipp lemma}. Therefore it remains to show the first identity.

    For each prime $p$, put $\bar{\mathcal{S}}_p=\mathbb{P}^{2d+1}(\Z/p\Z)$
    so that $\mathcal{S}_p=\mathbb{P}^{2d+1}(\Q)$. We have
    the affine lifting
    $$S_p=\{(C_1,\ldots,C_{2d+2})\in\Z^{2d+2}:\ p\nmid C_i\ \text{for some $i$}\}.$$
    Now $\mathcal{S}:=\bigcap_{p}\mathcal{S}_p=\mathbb{P}^{2d+1}(\Q)$ and its affine lifting is $S:=\bigcap_p S_p$. 

    Let $Z>0$ and put $S_{\leq Z}=\bigcap_{p\leq Z} S_p$. Since
    there are only finitely many congruence conditions in the definition
    of $S_{\leq Z}$ and since $S\subseteq S_{\leq Z}$, we have
    $$\limsup_{r_1,\ldots,r_{2d+2}\rightarrow\infty}\frac{\# S(\vec{r})}{\#\op{Box}(\vec{r})}\leq \lim_{r_1,\ldots,r_{2d+2}\rightarrow \infty}\frac{\# S_{\leq Z}(\vec{r})}{\#\op{Box}(\vec{r})}=\prod_{p\leq Z}\left(1-p^{-(2d+2)}\right).$$
    Let $Z\rightarrow\infty$, we have:
    \begin{equation}\label{eq:limsup leq}
    \limsup_{r_1,\ldots,r_{2d+2}\rightarrow\infty}\frac{\# S(\vec{r})}{\#\op{Box}(\vec{r})}\leq\frac{1}{\zeta(2d+2)}.
    \end{equation}

    Recall the set $T_{>Z}$ in Lemma~\ref{lem:T_>Z}. From the definitions of the involved sets, we have:
    $$S_{\leq Z}\setminus (S_{\leq Z}\cap T_{>Z})=S.$$
    Therefore $\# S(\vec{r})=\# S_{\leq Z}(\vec{r})-\# (S_{\leq Z}\cap T_{>Z})(\vec{r})$ and hence
    $$\liminf_{r_1,\ldots,r_{2d+2}\rightarrow\infty}\frac{\# S(\vec{r})}{\#\op{Box}(\vec{r})}=\prod_{p\leq Z}\left(1-p^{-(2d+2)}\right)-\limsup_{r_1,\ldots,r_{2d+2}\rightarrow\infty}\frac{\# (S_{\leq Z}\cap T_{>Z})(\vec{r})}{\#\op{Box}(\vec{r})}.$$
    Let $Z\rightarrow\infty$ and apply Lemma~\ref{lem:T_>Z}, we have
    \begin{equation}\label{eq:liminf=}
        \liminf_{r_1,\ldots,r_{2d+2}\rightarrow\infty}\frac{\# S(\vec{r})}{\#\op{Box}(\vec{r})}=\frac{1}{\zeta(2d+2)}.
    \end{equation}
    Combining \eqref{eq:S and calS}, \eqref{eq:limsup leq}, and \eqref{eq:liminf=}, we finish the proof of the first identity.
\end{proof}

\subsection{Some admissibility conditions}\label{subsec:admissibility}
Throughout this subsection, for each prime $p$ let $\mathcal{S}_p\subseteq \mathbb{P}^{2d+1}(\Q)$ be 
defined by a local condition, let $\mathcal{S}=\bigcap_p \mathcal{S}_p$, let $S_p$
be the affine lifting of $\mathcal{S}$, and let $S=\bigcap_p S_p$ be the affine lifting of $\mathcal{S}$. Let $\mathcal{S}_p'=\mathbb{P}^{2d+1}(\Q)\setminus \mathcal{S}_p$ and  $S_p'=\Z^{2d+2}\setminus S_p$. 
Warning: while $\mathcal{S}_p'$ is also defined by a local condition
at $p$ (since it is the preimage of $\mathbb{P}^{2d+1}(\Z/p^{n_p}\Z)\setminus\bar{\mathcal{S}}_p$), the 
set $S_p'$ is not the affine lifting of $\mathcal{S}_p'$. 
In fact, the affine lifting of $\mathcal{S}_p'$ is the set of points in $S_p'$ having at least one entry that is not divisible by $p$.
For $Z>0$, let
$\mathcal{S}_{>Z}'=\bigcup_{p>Z} \mathcal{S}_p'$ and 
$S_{>Z}'=\bigcup_{p>Z}S_p'$.

\begin{definition}\label{def:admissibility conditions}
    Let the symbol $\delta$ be either $\mathfrak{d}$, $\mu$, or $\mu_n$ where $n$ is an integer in $[1,2d+2]$. Note that $\delta^*$ denotes the corresponding affine version, $\underline{\delta}$ denotes the corresponding lower density, etc.
    We say that $\mathcal{S}$ is $\delta$-admissible if
    $$\lim_{Z\to\infty} \overline{\delta}(\mathcal{S}_{>Z}')=0.$$
    We say that $S$ is $\delta^*$-admissible if
    $$\lim_{Z\to\infty}\overline{\delta}^*(S_{>Z}')=0.$$
\end{definition}

\begin{proposition}\label{prop:symbol delta admissible}
    Let the symbol $\delta$ be either $\mathfrak{d}$, $\mu$, or $\mu_n$ where $n$ is an integer in $[1,2d+2]$. Suppose that either $\mathcal{S}$ is $\delta$-admissible or $S$ is $\delta^*$-admissible. Then we have
    \begin{equation}
        \delta(\mathcal{S})=\prod_p \pi(\mathcal{S}_p)\ \text{and}\ \delta^*(S)=\prod_p \pi(S_p).
    \end{equation}
\end{proposition}

\begin{remark}
    We interpret $\prod_p \pi(\mathcal{S}_p)$ as the limit of the sequence of partial products $\prod_{p\leq k}\pi(\mathcal{S}_p)$. This limit exists (with value in $[0,1]$) since $0\leq \pi(\mathcal{S}_p)\leq 1$ for every $p$. The same holds for 
    $\prod_p \pi(S_p)$. Our admissibility condition is analogous to the admissibility condition for Weierstrass models considered by Cremona and Sadek in \cite[Definition 4]{cremonasadek}.
\end{remark}

\begin{proof}
    Let $S_{\leq Z}=\bigcap_{p\leq Z} S_p$. Since there are only finitely many congruence conditions in the definition of $S_{\leq Z}$ and since $S\subseteq S_{\leq Z}$, we have:
    $$\overline{\delta}^*(S)\leq \overline{\delta}^*(S_{\leq Z})=\prod_{p\leq Z} \pi(S_p).$$
    Let $Z\rightarrow\infty$, we have:
    $$\overline{\delta}^*(S)\leq \prod_p \pi(S_p).$$

    From
    $$S_{\leq Z}\subseteq S\cup\left(\bigcup_{p>Z} S_p'\right)= S\cup S_{>Z}',$$
    we have
    $$S_{\leq Z}\setminus (S_{\leq Z}\cap S_{>Z}')=S$$
    and hence
    $$\# S(\vec{r})= \# S_{\leq Z}(\vec{r})-\# (S_{\leq Z}\cap S_{>Z}')(\vec{r}).$$
    Therefore
    \begin{equation}\label{eq:before 2 admissibility cases}
    \underline{\delta}^*(S)=\prod_{p\leq Z}\pi(S_p)-\overline{\delta}^*(S_{\leq Z}\cap S_{>Z}').
    \end{equation}
    
    First, consider the case when $S$ is $\delta^*$-admissible. From
    $$\delta^*(S_{\leq Z}\cap S_{>Z}')\leq \delta^*(S_{>Z}')$$
    and the $\delta^*$-admissibility condition, we have
    $$\lim_{Z\rightarrow\infty}\overline{\delta}^*(S_{\leq Z}\cap S_{>Z}')=0.$$

    Now we consider the case that $\mathcal{S}$ is $\delta$-admissible. The set
    $S_{\leq Z}\cap S_{>Z}'$ consists of two types of points:
    \begin{itemize}
        \item [(1)] The first consists of $Q=(C_1,\ldots,C_{2d+2})$ such that
        $\gcd(C_1,\ldots,C_{2d+2})=1$. These points come in pairs (namely $Q$ and $-Q$) and, according to Notation~\ref{notation} and the definitions of the involving sets, exactly one point in each pair belongs to
        $\mathcal{S}_{>Z}'$.

        \item [(2)] The second consists of $Q=(C_1,\ldots,C_{2d+2})$ such that
        $p\mid \gcd (C_1,\ldots,C_{2d+2})$ for some prime $p>Z$. We now recall the set
        $T_{>Z}$ in Lemma~\ref{lem:T_>Z}.
    \end{itemize}
    This gives
    $$\# (S_{\leq Z}\cap S_{>Z}')(\vec{r})\leq 2\#\mathcal{S}_{>Z}'(\vec{r})+\# T_{>Z}(\vec{r}).$$
    Lemma~\ref{lem:T_>Z} and the $\delta$-admissibility of $\mathcal{S}$ imply
    $$\lim_{Z\rightarrow\infty}\overline{\delta}^*(S_{\leq Z}\cap S_{>Z}')=0.$$

    Therefore in either case, we may let $Z\rightarrow\infty$ in \eqref{eq:before 2 admissibility cases} to have that
    $$\underline{\delta}^*(S)= \prod_{p}\pi(S_p).$$
    Combining this with the earlier upper bound for $\overline{\delta}^*(S)$, we have
    $$\delta^*(S)=\prod_p \pi(S_p).$$
    Finally, combining this with \eqref{eq:pi vs pi}, \eqref{eq:S and calS}, and Proposition~\ref{prop:box Schanuel and count Rd}, we obtain
    $$\delta(\mathcal{S})=\prod_p \pi(\mathcal{S}_p).$$
\end{proof}

\section{Minimal resultants}\label{s 4}

\par In this section, fix an integer $d\geq 2$ and our goal is to prove that a positive portion of
$\cR^{(d)}$ has globally minimal resultants. In fact, we consider a certain subset of $\cR^{(d)}$ that includes globally minimal rational maps that have good reduction at a given finite set of primes and prove that this subset has positive height density (for odd $d$) as well as positive weak box density (for any $d$). In section \ref{s 5}, we obtain an explicit lower bound for the weak box density in the case $d=2$. 

\subsection{The minimal resultant}
For a rational map $\phi(z)\in\Q(z)$ of degree $d$, we express $\phi(X/Y) = \frac{F(X, Y)}{G(X, Y)}$ with \[F(X,Y) = \sum_{i=0}^d A_i X^i Y^{d-i}\text{ and }
    G(X,Y) = \sum_{j=0}^d B_j X^j Y^{d-j},\]
where $(A_0, \dots, A_d, B_0, \dots, B_d)\in \cR^{(d)}\subseteq\mathbb{P}^{2d+1}(\Q)\subseteq \Z^{2d+2}$ according to Notation~\ref{notation}. We recall that this means $\gcd(A_0,\ldots,B_d)=1$
and the first non-zero coordinate is positive. Then we define $\op{Res}(\phi)$ to be $\op{Res}(F,G)$.

Let $p$ be a prime number and set $\widetilde{F}(X, Y)$ and $\widetilde{G}(X, Y)$ to denote the mod-$p$ reductions of $F(X, Y)$ and $G(X, Y)$ respectively, and set $\tilde{\phi}(X/Y):=\frac{\widetilde{F}(X, Y)}{\widetilde{G}(X, Y)}$. As is well known \cite[Theorem 2.15]{silvermantextbook}, the following conditions are equivalent
\begin{enumerate}
    \item $\op{deg}(\widetilde{\phi})=\op{deg}(\phi)$.
    \item The equations $\widetilde{F}(X, Y)=\widetilde{G}(X, Y)=0$ have no common solutions in $\mathbb{P}^1(\bar{\F}_p)$. 
    \item The resultant $\op{Res}(\phi)$ is not divisible by $p$.  
\end{enumerate}
We say that $\phi$ has good reduction at $p$ if one of the above conditions holds.

Let $\ord_p$ denote the usual $p$-adic valuation on $\Q$. The following definition
is taken from \cite[Chapter~4]{silvermantextbook}.
\begin{definition}
    \begin{itemize}
        \item Given $A\in \op{PGL}_2(\Q)$, set $\phi^A:=A^{-1}\circ \phi \circ A$. We set $\epsilon_p(\phi)$ to be the minimum value of $\op{ord}_p(\op{Res}(\phi^A))$, where $A$ ranges over $\op{PGL}_2(\Q)$.
        \item The \emph{minimal resultant} is defined to be the following product \[\mathfrak{R}_\phi:=\prod_p p^{\epsilon_p(\phi)}.\]
        \item The rational map $\phi$ is called \emph{minimal at $p$} if 
        $\op{ord}_p\left(\op{Res}(\phi)\right)=p^{\epsilon_p(\phi)}$.
        \item We say that $\phi$ is \emph{globally minimal} or $\phi$ has globally minimal resultant if it is minimal at every prime, equivalently $\mathfrak{R}_\phi=|\op{Res}(\phi)|$.
    \end{itemize}  
    \end{definition}

\begin{lemma}\label{d-1 lemma}
 The following assertions hold:
 \begin{enumerate}
     \item Suppose that for all primes $p$, $\op{ord}_p \left(\op{Res}(\phi)\right)\leq d-1$, then $\phi$ is minimal.
     \item Suppose that $d$ is odd and for all primes $p$, $\op{ord}_p \left(\op{Res}(\phi)\right)\leq 2d-1$, then $\phi$ is minimal.
 \end{enumerate}
   \end{lemma} 
   \begin{proof}
      The result follows immediately from \cite[Proposition 4.95]{silvermantextbook}.
   \end{proof}

We now introduce the relevant subsets of $\cR^{(d)}$ defined by local conditions. Let $\cN_p^{(k)}$ consist of $\phi\in\cR^{(d)}$ such that $p^k\nmid \op{Res}(\phi)$ so that $\cG_p:=\cN_p^{(1)}$ is the set of rational maps with good reduction at $p$.
\begin{definition}
    Let $\Sigma$ be a finite (possibly empty) set of prime numbers. With respect to notation above, set \[\begin{split}&\mathcal{V}_\Sigma:=\left(\bigcap_{p\in \Sigma} \cG_p\right)\cap \left(\bigcap_{p\notin \Sigma} \cN_p^{(d)}\right); \\ 
& \mathcal{W}_\Sigma:=\left(\bigcap_{p\in \Sigma} \cG_p\right)\cap \left(\bigcap_{p\notin \Sigma} \cN_p^{(2d)}\right).
    \end{split}\]
    In other words, $\V_\Sigma$ (resp. $\cW_\Sigma$) consists of rational maps $\phi\in\cR^{(d)}$ with $d$-power free (resp. $2d$-power free) resultant and good reduction at all primes $p\in \Sigma$. 
\end{definition}
\begin{remark}\label{rem:V and W have globally minimal maps}
   Let $\Sigma$ be a finite set of primes. It follows from Lemma \ref{d-1 lemma} that all $\phi\in \V_\Sigma$ are globally minimal. Moreover, if $d$ is odd, then all $\phi\in \cW_{\Sigma}$ are globally minimal.
\end{remark}
\subsection{Equations over finite fields}
\par In this subsection, we shall prove bounds for $\pi(\cG_p)$. For this purpose we recall the following well-known results about the number of solutions to equations over finite fields. Throughout, $q$ will be a power of a prime $p$ and $\F_q$ shall denote the field with $q$ elements. 

\begin{theorem}[Warning's theorem]\label{warning}
    Let $h\in \F_q[x_1, \dots, x_n]$ be a polynomial with degree $D<n$. Then the number of solutions $\alpha=(\alpha_1, \dots, \alpha_n)\in \F_q^n$ to $h(\alpha)=0$ is divisible by $p$. 
\end{theorem}
\begin{proof}
    See \cite[Theorem 6.5]{lidlnied}.
\end{proof}

\begin{theorem}\label{lower bound for h theorem}
    Let $h$ be as in Theorem \ref{warning} and let $N$ be the number of solutions $\alpha=(\alpha_1, \dots, \alpha_n)\in \F_q^n$ to $h=0$. Assume that $N\geq 1$. Then, we have that $N\geq q^{n-D}$. 
\end{theorem}
\begin{proof}
    See \cite[Theorem 6.11]{lidlnied}.
\end{proof}

\begin{theorem}\label{d q n-1}
    Let $h\in \F_q[x_1, \dots, x_n]$ be a non-zero polynomial with degree $D$. Then the number of solutions to $h=0$ is at most $D q^{n-1}$. 
\end{theorem}
\begin{proof}
 See \cite[Theorem 6.13]{lidlnied}. 
\end{proof}

We will apply the above theorems to the generic resultant of two polynomials of degree $d$.
\begin{definition}\label{def:generic resultant}
    Consider two generic homogeneous polynomials of degree $d$:
    $$F(X,Y)=\sum_{i=0}^d x_iX^iY^{d-i}\ \text{and}\ G(X,Y)=\sum_{j=0}^{d} y_jX^jY^{d-j}.$$
    Let $\mathfrak{r}(x_0,\ldots,x_d,y_0,\ldots,y_d)\in\Z[x_0,\ldots,y_d]$ be the resultant of $F$ and $G$. This polynomial is homogeneous of degree $2d$ and has degree $d$ in each of the variables $x_i$'s and $y_j$'s.
\end{definition}

The set $\cG_p$ is defined by a local condition at $p$ since it is the preimage
of
$$\bar{\cG}_p:=\{P\in\mathbb{P}^{2d+1}(\Z/p\Z):\ \mathfrak{r}(P)\neq 0\}$$
under the reduction map $\mathbb{P}^{2d+1}(\Q)\rightarrow\mathbb{P}^{2d+1}(\Z/p\Z)$. The affine lifting of $\bar{\cG}_p$ is
$$\bar{G}_p:=\{Q\in(\Z/p\Z)^{2d+2}:\ \mathfrak{r}(Q)\neq 0\}.$$
The preimage $G_p$ of $\bar{G}_p$ under the reduction map
$\Z^{2d+2}\rightarrow (\Z/p\Z)^{2d+2}$ is the affine lifting of $\cG_p$. 
The complement $\bar{G}_p':=(\Z/p\Z)^{2d+2}\setminus\bar{G}_p$ is the zero set of
$\mathfrak{r}$ in $(\Z/p\Z)^{2d+2}$. By \eqref{eq:pi vs pi}, we have
\begin{equation}\label{eq:pi(cG_p)}
    \pi(\cG_p)=\frac{\#\bar{G}_p}{p^{2d+2}-1}.
\end{equation}

\begin{lemma}\label{2d p to the 2d+1}
     We have:
     \begin{enumerate}
         \item $\#\bar{G}_p$ and $\#\bar{G}_p'$ are divisible by $p$,
         \item $\#\bar{G}_p'\geq p^2$,
         \item $\#\bar{G}_p'\leq 2d p^{2d+1}$, or equivalently, $\#\bar{G}_p\geq p^{(2d+2)}-2d p^{2d+1}$, 
         \item $\# \bar{G}_p\geq p^2$. 
     \end{enumerate}
 \end{lemma}
\begin{proof}
     We note that $\#\bar{G}_p'=p^{2d+2}-\#\bar{G}_p$. Then part (1) follows from Theorem \ref{warning}. For part (2), the lower bound $\#\bar{G}_p'\geq p^{2d+2-2d}=p^2$ follows from Theorem \ref{lower bound for h theorem}. Next, part (3) follows from Theorem \ref{d q n-1}. Finally for part (4), since $Y^d$ and $X^d$ have no common root in $\mathbb{F}_p$, we have that
     $\mathfrak{r}(1,0,\ldots,0,1)=:c$ is a non-zero element of $\mathbb{F}_p$. Then we apply  Theorem~\ref{lower bound for h theorem} for the polynomial $\mathfrak{r}-c$. 
\end{proof}

\begin{corollary}\label{cor:pi Gp lemma}
    The following assertions hold
    \begin{enumerate}
        \item $\displaystyle\pi(\cG_p)\geq \frac{1}{p^{2d}}$ for all primes $p$.
        \item For $p>2d$, we have that $\displaystyle\pi(\cG_p)\geq 1-\frac{2d}{p}$.
    \end{enumerate}
\end{corollary}
\begin{proof}
    This is from \eqref{eq:pi(cG_p)}, part (3), and part (4) of Lemma~\ref{2d p to the 2d+1}.
\end{proof}

\subsection{Some positive densities}
\par Throughout this subsection, let $p$ be a prime number 
outside $\Sigma$ and let $\gamma_1,\gamma_2,\ldots$ denote positive constants depending only on $d$. First, we study the quantities $\pi(\cN_p^{(2)})$. Then we establish the corresponding admissibility condition for $\V_{\Sigma}$ and $\cW_{\Sigma}$. Finally we prove the first two theorems stated in the Introduction.

For $1\leq k\leq 2d$, the set $\cN_p^{(k)}$ is defined by a local condition
since it is the preimage of
$$\bar{\cN}_p^{(k)}:= \{P\in\mathbb{P}^{2d+1}(\Z/p^k\Z):\ \mathfrak{r}(P)\neq 0\ \text{in}\ \Z/p^k\Z\}$$
under the reduction map $\mathbb{P}^{2d+1}(\Q)\rightarrow\mathbb{P}^{2d+1}(\Z/p^k\Z)$. Hence the sets $\V_{\Sigma}$ and $\cW_{\Sigma}$ are defined
by local conditions. 
The affine lifting of $\bar{\cN}_p^{(k)}$ is
$$\bar{N}_p^{(k)}:=\{Q\in(\Z/p^k\Z)^{2d+2}:\ \mathfrak{r}(Q)\neq 0\ \text{in}\ \Z/p^k\Z\}.$$
We emphasize that since $k\leq 2d$ and $\mathfrak{r}$ is homogeneous of 
degree $2d$, the condition $\mathfrak{r}(Q)\neq 0$ in $\Z/p^k\Z$
guarantees that at least one of the entries of $Q$ is a unit of $\Z/p^k\Z$. 
The preimeage $N_p^{(k)}$ of $\bar{N}_p^{(k)}$ under the reduction map
$\Z^{2d+2}\rightarrow (\Z/p^k\Z)^{2d+2}$ is the affine lifting of
$\cN_p^{(k)}$. 
From \eqref{eq:pi vs pi}, we have:
\begin{equation}\label{eq:pi}
    \pi(\cN_p^{(k)})=\frac{\#\bar{N}_p^{(k)}}{p^{(2d+2)k}(1-p^{-(2d+2)})}.
\end{equation}
For $1\leq k_1\leq k_2\leq 2d$, each element of $\bar{N}_p^{(k_1)}$ lifts to
$p^{(k_2-k_1)(2d+2)}$ many elements of $\bar{N}_p^{(k_1)}$. Therefore:
\begin{equation}\label{eq:picNpk1 < picNpk2}
    \pi(\cN_p^{(k_1)})\leq \pi(\cN_p^{(k_2)}). 
\end{equation}

The complement $(\cN_p^{(k)})':=\mathbb{P}^{2d+1}(\Q)\setminus\cN_p^{(k)}$
is defined by a local condition since it is the preimage of
$$(\bar{\cN}_p^{(k)})':=\mathbb{P}^{2d+1}(\Z/p^k\Z)\setminus\bar{\cN}_p^{(k)}= \{P\in\mathbb{P}^{2d+1}(\Z/p^k\Z):\ \mathfrak{r}(P)=0\ \text{in}\ \Z/p^k\Z\}.$$
We have:
\begin{equation}\label{eq:piNk1'>piNk2'}
\pi((\cN_p^{(k_1)})')=1-\pi(\cN_p^{(k_1)})\geq 1-\pi(\cN_p^{(k_2)})=\pi((\cN_p^{(k_2)})').
\end{equation}

The complement $(\bar{N}_p^{(k)})':=(\Z/p^k\Z)^{2d+2}\setminus\bar{N}_p^{(k)}$ is the zero set of 
$\mathfrak{r}$ in $(\Z/p^k\Z)^{2d+2}$ and the complement $(N_p^{(k)})':=\Z^{2d+2}\setminus N_p^{(k)}$ is the preimage of $(\bar{N}_p^{(k)})'$ under the reduction map
$\Z^{2d+2}\rightarrow(\Z/p^k\Z)^{2d+2}$. As in Subsection~\ref{subsec:admissibility}, we warn the reader that neither $(\bar{N}_p^{(k)})'$ is the affine lifting of $(\bar{\cN}_p^{(k)})'$
nor $(N_p^{(k)})'$ is the affine lifting of $(\cN_p^{(k)})'$ (roughly speaking, this means taking the complement and taking the affine lifting do not commute to each other). In fact, the affine lifting of $(\bar{\cN}_p^{(k)})'$ is the set of points in $(\bar{N}_p^{(k)})'$ having at least one entry that is a unit of $\Z/p^k\Z$. Similarly, the affine lifting of
$(\cN_p^{(k)})'$ is the set of points in $(N_p^{(k)})'$ having at least one entry that is not divisible by $p$. Therefore:
\begin{equation}\label{eq:pi cNpk complement}
    \pi((\cN_p^{(k)})')\leq \frac{\# (\bar{N}_p^{(k)})'}{p^{(2d+2)k}(1-p^{-(2d+2)})}.
\end{equation}

In order to obtain an upper bound for $\# (\bar{N}_p^{(k)})'$, we use the following result
\begin{proposition}\label{f irred}
    The generic resultant polynomial $\mathfrak{r}(x_0,\ldots,y_d)\in\Z[x_0,\ldots,y_d]$ is absolutely irreducible. 
\end{proposition}
\begin{proof}
    This result follows from \cite[Ch.3, Proposition 3.1 and Example 3.6]{GKZ}.
\end{proof}

\begin{notn}
For the rest of this section, let $\mathfrak{s}$ denote one of the partial derivatives of $\mathfrak{r}$ (when necessary, which partial derivative will be specified precisely in later context). This  is a homogeneous polynomial of degree $2d-1$ and has degree $d-1$ in the variable for which we differentiate.    
\end{notn}

\begin{proposition}\label{Cp to the 2d}
    $\# \{P\in \F_p^{2d+2}\mid \mathfrak{r}(P)=\mathfrak{s}(P)=0\}\leq \gamma_1 p^{2d}$. 
\end{proposition}
\begin{proof}
Consider the algebra $\mathfrak{B}:=\Q[x_0, \dots, x_d, y_0, \dots, y_d]/(\mathfrak{r},\mathfrak{s})$ and let $\bar{x}_i$ and $\bar{y}_j$ denote the image of
$x_i$ and $y_j$ in $\mathfrak{B}$. Since $\mathfrak{r}$ is irreducible and does not divide
$\mathfrak{s}$, the dimension of $\mathfrak{B}$ is $D\leq 2d$. By the Noether normalization lemma, 
there exist $X_1,\ldots,X_D\in \mathfrak{B}$ that are algebraically independent over $\Q$ such that 
$\mathfrak{B}$ is integral over $\Q[X_1, \dots, X_{D}]$. Thus there is a natural number $N$ such that the $\bar{x}_i$'s and $\bar{y}_j$'s are all  integral over $\Z[1/N, X_1, \dots, X_{D}]$. Denote by $\gamma_1'$ the product of the degrees of the minimal polynomials of $\bar{x}_i$ and $\bar{y}_j$ over $\Z[1/N, X_1, \dots, X_{D}]$.

\par Now suppose $p\nmid N$, and $(X_1, \dots, X_{D})\in (\Z/p\Z)^D$. Then each of the variables $x_i$ and $y_j$ satisfies a monic polynomial equation over $\Z/p\Z$. Let $g(z)\in \Z/p\Z[z]$ be their product and suppose that \[P=(A_0, \dots, A_d, B_0, \dots, B_d)\in (\Z/p\Z)^{2d+2}\] is a point for which $\mathfrak{r}(P)=\mathfrak{s}(P)=0$. Then each of the coordinates $A_i$ and $B_j$ is a solution to $g(z)=0$. Since the degree of $g(z)$ is bounded by the constant $\gamma_1'$ which is independent of the specialization $(X_1, \dots, X_D)$ and the prime $p$, we find that there are at most $(\gamma_1')^{2d+2}$ choices of the point $P$ such that $\mathfrak{r}(P)=\mathfrak{s}(P)=0$. Thus for $p\nmid N$, the number of points $P$ is at most $(\gamma_1')^{2d+2} p^{2d}$. Setting $\gamma_1:=\op{max}\{(\gamma')^{2d+2},N^2\}$, we find that the inequality in the statement of the proposition holds trivially for $p\mid N$. This finishes the proof. 
\end{proof}
\begin{remark}
    In the above proof, suppose $\bar{x}_i$ (resp. $\bar{y}_j$) satisfies an integral dependence relation of degree $a_i$ (respectively $b_j$) over 
    $\Z[1/N,X_1,\ldots,X_D]$
    then we may take $\gamma_1=\max(N^2,a_1\cdots a_db_1\cdots b_d)$.
\end{remark}

The following  Hensel's lifting lemma is well-known, we include a simple proof for the convenience of the reader.
\begin{lemma}\label{partial derivative lemma}
    Let $h\in \Z_p[x_1, \dots, x_n]$ be a polynomial and $\alpha\in (\Z/p\Z)^n$. Assume that $h(\alpha)=0\bmod p$ and $\frac{\partial h}{\partial x_i}(\alpha)\neq 0\bmod p$ for some $i\in [1,n]$. Then for every $k\geq 1$, the element $\alpha$ lifts to $p^{(n-1)(k-1)}$ many values $\beta\in (\Z/p^k\Z)^n$ for which $h(\beta)=0\bmod p^k$.
\end{lemma}
\begin{proof}
We prove this by induction on $k$; the case $k=1$ is immediate. Suppose the statement holds for $k$. Let $\beta\in(\Z/p^k\Z)^n$ be one of the $p^{(n-1)(k-1)}$ many liftings of $\alpha$
with $h(\beta)=0\bmod p^k$. Let $(b_1,\ldots,b_n)\in\Z^n$ be a preimage of $\beta$ under the reduction map $\Z^n\rightarrow (\Z/p^k\Z)^n$. We have the Taylor series expansion:
\[h(x_1, \dots, x_n)=h(b_1,\ldots,b_n)+\sum_{i=1}^n \frac{\partial h }{\partial x_i}(b_1,\ldots,b_n)(x_i-b_i)+\text{higher order terms}.\]

Liftings of $\beta$ in $(\Z/p^{k+1}\Z)^n$ are the images of integer tuples 
of the form
$(b_1+t_1p^k,\ldots,b_n+t_np^k)$ with $t_1,\ldots,t_n\in\{0,\ldots,p-1\}$. We have
$h(b_1+t_1p^k,\ldots,b_n+t_np^k)=0\bmod p^{k+1}$ if and only if
$$\frac{h(b_1,\ldots,b_n)}{p^k}+\sum_{i=1}^n \frac{\partial h }{\partial x_i}(b_1,\ldots,b_n)t_i=0\bmod p.$$
Since $\frac{\partial h}{\partial x_i}(b_1,\ldots,b_n)\neq 0\bmod p$ for some $i$, there are 
exactly $p^{n-1}$ many solutions $(t_1,\ldots,t_n)$. Hence, overall there are 
$p^{(n-1)k}$ many desired liftings of $\alpha$ in $(\Z/p^{k+1}\Z)^n$.
\end{proof}

\begin{proposition}\label{prop:gamma2}
     Let $2\leq k\leq 2d$, we have $\displaystyle\pi((\cN_p^{(k)})')\leq \frac{\gamma_2}{p^2}$.
\end{proposition}
\begin{proof}
From \eqref{eq:piNk1'>piNk2'}, it suffices to show $\pi((\cN_p^{(2)})')\leq \gamma_2/p^2$.  We define subsets $S_1$ and $S_2$ of $(\Z/p\Z)^{2d+2}$ as follows. 
\begin{itemize}
    \item The set $S_1$ consists of points for which $\mathfrak{r}=0$ and $\mathfrak{s}=0$. 
    \item The set $S_2$ consists of points for which $\mathfrak{r}=0$ and $\mathfrak{s}\neq 0$. 
\end{itemize}
Each element of $S_1$ lifts to $p^{2d+2}$ many elements of $(\Z/p^2\Z)^{2d+2}$ and we have
$\#S_1\leq \gamma_1p^{2d}$ by Proposition~\ref{Cp to the 2d}. Each element of
$S_2$ lifts to $p^{2d+1}$ many elements $\beta\in(\Z/p^2\Z)^{2d+2}$ such that
$\mathfrak{r}(\beta)=0\bmod p^2$ by Lemma~\ref{partial derivative lemma} and we have $\#S_2\leq 2dp^{2d+1}$ by Theorem~\ref{d q n-1}. Overall, we have:
$$\# (\cN_p^{(2)})'\leq (\gamma_1+2d)p^{4d+2}.$$
Then we apply \eqref{eq:pi cNpk complement} to get the desired upper bound on $\pi((\cN_p^{(2)})')$.
\end{proof}

The set $W_{\Sigma}:=\left(\bigcap_{p\in\Sigma} G_p\right)\cap\left(\bigcap_{p\notin\Sigma} N_p^{(2d)}\right)$ is the affine lifting of $\cW_{\Sigma}$. 
Recall the complement $(N_p^{(2d)})'=\Z^{2d+2}\setminus N_p^{(2d)}$. 
We recall the notation at the beginning of Subsection~\ref{subsec:admissibility}. Let $Z$ be a positive real number such that
$Z>\max\{p\in\Sigma\}$ and let 
$$W_{>Z}'=\bigcup_{p>Z} (N_p^{(2d)})'$$
(we drop the subscript $\Sigma$ here to simplify the notation slightly). We have:
\begin{proposition}\label{delta admissible proposition}
    The set $W_{\Sigma}$ is $\mathfrak{d}^*$-admissible (cf.~Definition \ref{def:admissibility conditions}). Therefore 
    $$\mathfrak{d}(\cW_\Sigma)=\prod_{p\in\Sigma}\pi(\cG_p)\prod_{p\notin\Sigma}\pi(\cN_p^{(2d)}).$$
\end{proposition}
\begin{proof}
    The second assertion follows from the first and Proposition~\ref{prop:symbol delta admissible}.  For $x>1$, $W_{>Z}'(x)$
    is the set of integer tuples $(A_0,\ldots,A_d,B_0,\ldots,B_d)\in [-x,x]^{2d+2}$ such that 
$$p^{2d}\mid \mathfrak{r}(A_0,\ldots,B_d)$$
for some prime $p>Z$. 
    We need to show that 
    \[\lim_{Z\rightarrow \infty} \left(\limsup_{x\rightarrow \infty} \frac{\# W_{>Z}'(x)}{(2x)^{2d+2}}\right)=0.\] 
    
    Let $\displaystyle\mathfrak{s}=\frac{\partial\mathfrak{r}}{\partial x_0}$. We partition $W_{>Z}'(x)$ into disjoint sets $S_{>Z}(x)$ and $T_{>Z}(x)$, where $S_{>Z}(x)$ consists of $(A_0, \dots, A_d, B_0, \dots, B_d)\in W_{>Z}'(x)$ such that there exists $p>Z$ for which 
    \[p^{2d}\mid\mathfrak{r}(A_0, \dots, B_d)\text{ and }p\mid\mathfrak{s}(A_0, \dots, B_d)\] and 
    \[T_{>Z}(x):=W_{>Z}'(x)\setminus S_{>Z}(x).\]
    It follows from \cite[Lemma~5.1]{Poonensquarefree} that 
    \[\lim_{Z\rightarrow\infty}\left(\limsup_{x\rightarrow \infty} \frac{\# S_{>Z}(x)}{(2x)^{2d+2}}\right)=0.\]
    Now we aim to show that
    \[\lim_{Z\rightarrow \infty} \left(\limsup_{x\rightarrow \infty} \frac{\# T_{>Z}(x)}{(2x)^{2d+2}}\right)=0.\]
    
    Let $Q=(A_1, \dots, A_d, B_0, \dots, B_{d})\in \Z^{2d+1}\cap [-x, x]^{2d+1}$ and consider the $1$-variable polynomial $h_Q(u)$ of degree at most $d$ defined by 
    \[h_Q(u):=\mathfrak{r}(u, A_1, \dots, A_d, B_0, \dots , B_{d-1}, B_d)\in \Z[u],\] and note that 
    \[h_Q'(u)=\mathfrak{s}(u, A_1, \dots, A_d, B_0, \dots, B_{d-1},B_d).\]
    We count the number of choices of $A_0\in \Z\cap [-x, x]$ such that $p^{2d}| h_Q(A_0)$ and $p\nmid h_Q'(A_0)$ for some $p>Z$. We may assume that $h_Q(u)$ and $h_Q'(u)$ are non-zero polynomials; otherwise there does not exist any such $A_0$. Once $Q$ and $p$ are fixed, the number of residue classes of $A_0\mod{p}$ for which $p|h_Q(A_0)$ is $\leq \deg(h_Q)\leq d$. On the other hand, since $p\nmid h_Q'(A_0)$, it follows from Lemma~\ref{partial derivative lemma}   that each residue class $A_0\mod{p}$ lifts to a unique residue class $A_0\mod{p^{2d}}$ such that $p^{2d}|h_Q(A_0)$. Thus, there are $\leq d$ residue classes for $A_0\mod{p^{2d}}$ such that $p^{2d}|h_Q(A_0)$ and $p\nmid h_Q'(A_0)$. Thus, the number of values of $A_0$ in $\Z\cap [-x, x]$ is at most $\displaystyle d\left(\frac{2x}{p^{2d}}+1\right)$. 
    
    Since 
    $(A_0, \dots, B_d)\in [-x, x]^{2d+2}$ and $\mathfrak{r}$ has degree $2d$ it follows that 
    \[|\mathfrak{r}(A_0, \dots, B_d)|\leq \gamma_3 x^{2d}.\] There are at most $d$ many values of $A_0$ for which $h_Q(A_0)=0$. When $h_Q(A_0)\neq 0$, we have: 
    \[p^{2d}\leq |\mathfrak{r}(A_0, \dots, B_d)|\leq \gamma_3 x^{2d}.\]
    Thus, we find that $Z<p\leq \sqrt[2d]{\gamma_3} x$. The total number of choices for $Q$ is $< (2x+2)^{2d+1}<2\cdot (2x)^{2d+1}$ when $x$ is sufficiently large. We have shown that \[\begin{split} & \# T_{>Z}(x) \\ 
    \leq & \sum_{Z<p< \sqrt[2d]{\gamma_3}x} 2\cdot(2x)^{2d+1}\cdot\left(\frac{4d x}{p^{2d}}+ 2d\right) \\  
    =& O\left(x^{2d+2} \sum_{Z<p} \frac{1}{p^{2d}}\right)+O\left(x^{2d+1}\sum_{Z<p< \sqrt[2d]{\gamma_3}x} 1\right) \\
    = & O\left(x^{2d+2} \left(\sum_{Z<p} \frac{1}{p^{2d}}+\frac{1}{\log x}\right)\right). \end{split}\]
    Hence, we find that 
    \[\limsup_{x\rightarrow \infty} \frac{\#T_{>Z}(x)}{(2x)^{2d+2}}= O\left( \sum_{Z<p} \frac{1}{p^{2d}}\right),\] and conclude that
     \[\lim_{Z\rightarrow \infty} \left(\limsup_{x\rightarrow \infty} \frac{\# T_{>Z}(x)}{(2x)^{2d+2}}\right)=0.\]
     This completes the proof of the result. 
\end{proof}

The set $V_{\Sigma}:=\left(\bigcap_{p\in\Sigma} G_p\right)\cap\left(\bigcap_{p\notin\Sigma} N_p^{(d)}\right)$ is the affine lifting of $\V_{\Sigma}$. Let 
$$V_{>Z}'=\bigcup_{p>Z} (N_p^{(d)})'.$$
\begin{proposition}\label{mu admissible proposition}
    Let $n$ be an integer in $[1,2d+2]$. The set $V_{\Sigma}$ is $\mu_n^*$-admissible (cf.~Definition~\ref{def:admissibility conditions}). Therefore
    $$\mu_n(\V_{\Sigma})=\prod_{p\in\Sigma}\pi(\cG_p)\prod_{p\notin\Sigma}\pi(\cN_p^{(d)}).$$
\end{proposition}
\begin{proof}
    The second assertion follows from the first and Proposition~\ref{prop:symbol delta admissible}.  For $\vec{r}=(r_1,\ldots,r_{2d+2})$ a tuple of (sufficiently large) positive real numbers, $V_{>Z}'(\vec{r})$
    is the set of integer tuples $(C_1,\ldots,C_{2d+2})=(A_0,\ldots,A_d,B_0,\ldots,B_d)\in \op{Box}(\vec{r})$ such that 
$$p^{d}\mid \mathfrak{r}(C_1,\ldots,C_{2d+2})$$
for some prime $p>Z$. We need to show that 
    \[\lim_{Z\rightarrow \infty} \left(\limsup_{r_1, \dots, \hat{r}_n, \dots, r_{2d+2}\rightarrow\infty} \limsup_{r_{n}\rightarrow \infty}\frac{\# V_{>Z}'(\vec{r})}{\# \op{Box}(\vec{r})}\right)=0.\] 
    
    The proof follows a similar pattern to the proof of the previous proposition. Let 
    $\displaystyle\mathfrak{s}$ be the partial derivative 
    of $\mathfrak{r}$ with respect to the $n$-th variable. 
    We partition $V_{>Z}'(\vec{r})$ into disjoint sets $S_{>Z}(\vec{r})$ and $T_{>Z}(\vec{r})$, where $S_{>Z}(\vec{r})$ consists of $(C_1, \dots, C_{2d+2})\in V_{>Z}'(\vec{r})$ such that there exists $p>Z$ for which 
    \[p^d\mid\mathfrak{r}(C_1, \dots, C_{2d+2})\text{ and }p\mid\mathfrak{s}(C_1, \dots, C_{2d+2})\] and 
    \[T_{>Z}(\vec{r}):=V_{>Z}'(\vec{r})\setminus S_{>Z}(\vec{r}).\]

    It follows from \cite[Lemma~5.1]{Poonensquarefree} that
    $$\lim_{Z\rightarrow \infty} \left(\limsup_{r_1, \dots, r_{2d+2}\rightarrow\infty} \frac{\# S_{>Z}(\vec{r})}{\# \op{Box}(\vec{r})}\right)=0,\ \text{hence}$$
    $$\lim_{Z\rightarrow \infty} \left(\limsup_{r_1, \dots, \hat{r}_n, \dots, r_{2d+2}\rightarrow\infty} \limsup_{r_{n}\rightarrow \infty} \frac{\# S_{>Z}(\vec{r})}{\# \op{Box}(\vec{r})}\right)=0.$$
    Now we aim to show that
    $$\lim_{Z\rightarrow \infty} \left(\limsup_{r_1, \dots, \hat{r}_n, \dots, r_{2d+2}\rightarrow\infty} \limsup_{r_{n}\rightarrow \infty} \frac{\# T_{>Z}(\vec{r})}{\# \op{Box}(\vec{r})}\right)=0.$$
    
    \par Partition $T_{>Z}(\vec{r})$ into disjoint sets $A_{>Z}(\vec{r})$ and $B_{>Z}(\vec{r})$, where $A_{>Z}(\vec{r})$ consists of $(C_1, \dots, C_{2d+2})\in T_{>Z}(\vec{r})$ for which 
    \[\mathfrak{r}(C_1, \dots, C_{2d+2})=0.\] It follows from Lemma~\ref{Schzipp lemma} that 
    \[\lim_{r_1, \dots, r_{2d+2}\rightarrow \infty} \frac{\#A_{>Z}(\vec{r}) }{\# \op{Box}(\vec{r})}=0.\]
   Thus it remains to show that 
    \[\lim_{Z\rightarrow \infty}\left(\limsup_{r_1, \dots, \hat{r}_n, \dots, r_{2d+2}\rightarrow\infty} \limsup_{r_{n}\rightarrow \infty}  \frac{\#B_{>Z}(\vec{r}) }{\# \op{Box}(\vec{r})}\right)=0.\]
    
    Let \[Q=(C_1, \dots, C_{n-1}, C_{n+1}, \dots, C_{2d+2})\in \op{Box}(r_1, \dots, r_{n-1}, r_{n+1}, \dots, r_{2d+2})\cap \Z^{2d+1}\] and consider the $1$-variable polynomial $h_Q(x)$ of
    degree at most $d$ defined by 
    \[h_Q(x):=\mathfrak{r}(C_1, \dots, C_{n-1}, x, C_{n+1}, \dots , C_{2d+2})\in \Z[x],\] and note that 
    \[h_Q'(x)=\mathfrak{s}(C_1, \dots, C_{n-1}, x, C_{n+1}, \dots , C_{2d+2}).\]  Given $r_{n}>0$, 
    from the definition of $S_{>Z}(\vec{r})$, $T_{>Z}(\vec{r})$, $A_{>Z}(\vec{r})$, and 
    $B_{>Z}(\vec{r})$, 
    we need to count the number of $C_n\in \Z\cap [-r_{n}, r_{n}]$ such that $p^d\mid h_Q(C_n)$, $h_Q(C_n)\neq 0$, and $p\nmid h_Q'(C_n)$. We may assume that $h_Q(x)$ and $h_Q'(x)$ are non-zero polynomials; otherwise there does not exist such a $C_n$. 
    
    Once $Q$ is fixed, the number of residue classes of $C_n\mod{p}$ for which $p\mid h_Q(x)$ is $\leq d$. On the other hand, since $p\nmid h_Q'(C_n)$, it follows from Lemma~\ref{partial derivative lemma} that each residue class $C_n\mod{p}$ lifts to a unique residue class $C_n\mod{p^d}$ such that $p^d\mid h_Q(C_n)$. Thus, there are $\leq d$ residue classes $C_n\mod{p^d}$ such that $p^d\mid h_Q(C_n)$ and $p\nmid h_Q'(C_n)$. Thus, the number of values of $C_n$ in $\Z\cap [-r_{n}, r_{n}]$ is at most $\displaystyle d\left(\frac{2r_{n}}{p^d}+1\right)$. 
    
    There exists a positive constant $\Gamma$ depending only on $d$ and the $r_i$s with $1\leq i\leq 2d+2$ and $i\neq n$ such that 
    \[|h_Q(x)|\leq \Gamma r_n^d\] 
    for every $x\in[-r_n,r_n]$.
    On the other hand, since $p^d\mid h_Q(C_n)$ 
    and $h_Q(C_n)\neq 0$, we have 
    \[p^d\leq |h_Q(C_n)|\leq \Gamma r_n^d.\]
    Thus, we find that $Z<p\leq \Gamma^{1/d} r_{n}$. The total number of choices for $Q$ is less than $2\cdot\prod_{i\neq n} (2r_i)$ when the $r_i$s are sufficiently large. We have shown that \[\begin{split}\#B_{>Z}(\vec{r})\leq & \sum_{Z<p\leq \Gamma^{1/d} r_n} \left(2\cdot\prod_{i\neq n} (2r_i)\left(\frac{2d r_{n}}{p^d}+ d\right)\right)  \\
    \leq & 2d\prod_{i=1}^{2d+2} (2r_i)\sum_{Z<p} \frac{1}{p^d}+ 2d\prod_{i=1}^{2d+2} (2r_i) \sum_{Z<p\leq \Gamma^{1/d}r_n} \frac{1}{2r_n}\\
     \leq & 2d\prod_{i=1}^{2d+2} (2r_i)\sum_{Z<p} \frac{1}{p^d}+ 2d\prod_{i=1}^{2d+2} (2r_i) \frac{\Gamma^{1/d}}{\log(\Gamma^{1/d}r_{n})},\\
    \end{split}\]
    where we use the fact that the number of primes $Z<p\leq \Gamma^{1/d}r_n$ is less than
    $2\Gamma^{1/d}r_n/\log(\Gamma^{1/d}r_n)$ when $r_n$ is large.

     We recall that $\Gamma$ is independent of $r_n$. The above estimate implies
    $$\limsup_{r_1, \dots, \hat{r}_n, \dots, r_{2d+2}\rightarrow\infty} \limsup_{r_{n}\rightarrow \infty}\frac{\#B_{>Z}(\vec{r})}{\# \op{Box}(\vec{r})}\leq 2d\sum_{Z<p}\frac{1}{p^d}.$$
    Therefore
    \[\lim_{Z\rightarrow \infty}\left(\limsup_{r_1, \dots, \hat{r}_n, \dots, r_{2d+2}\rightarrow\infty} \limsup_{r_{n}\rightarrow \infty}  \frac{\#B_{>Z}(\vec{r}) }{\# \op{Box}(\vec{r})}\right)=0\]
    and we finish the proof. 
\end{proof}

\begin{notn}
   Let $\Sigma_1=\{p\in\Sigma:\ p<2d\}$, $\Sigma_2=\Sigma\setminus\Sigma_1$, and $\Omega=\{p\notin\Sigma:\ p^2\leq \gamma_2\}$. 
\end{notn}

    \begin{theorem}\label{main thm of section 4 Wsigma}
        We have the following positive lower bound on $\mathfrak{d}(\cW_{\Sigma})$:
        \begin{equation}\label{d(cSSigma}\mathfrak{d}(\cW_\Sigma)\geq \prod_{p\notin \Sigma\cup \Omega}\left(1-\frac{\gamma_2}{p^2}\right)\times \left(\prod_{p\in \Omega\cup \Sigma_1}\frac{1}{p^{2d}}\right)\times \prod_{p\in \Sigma_2}\left(\frac{p^{(2d+2)}-2d p^{2d+1}}{p^{(2d+2)}-1}\right).\end{equation}
    \end{theorem}
    \begin{proof}
        For $p\in\Sigma_2$, from Lemma~\ref{2d p to the 2d+1} and \eqref{eq:pi(cG_p)}, we have:
        $$\pi(\cG_p)\geq \frac{p^{(2d+2)}-2d p^{2d+1}}{p^{(2d+2)}-1}.$$
        For $p\in \Omega\cup\Sigma_1$, from Lemma~\ref{2d p to the 2d+1}, \eqref{eq:pi(cG_p)}, and
        \eqref{eq:picNpk1 < picNpk2}, we have:
        $$\frac{1}{p^{2d}}\leq \pi(\cG_p)\leq \pi(\cN_p^{(2d)}).$$
        For $p\notin\Sigma\cup\Omega$, from Proposition~\ref{prop:gamma2}, we have:
        $$\pi(\cN_p^{(2d)})=1-\pi((\cN_p^{(2d)})')\geq 1-\frac{\gamma_2}{p^2}.$$
        Combining the above inequalities with Proposition~\ref{delta admissible proposition}, we obtain the desired result.
    \end{proof}

\begin{theorem}\label{main thm of section 4}
        Let $n$ be an integer in $[1,2d+2]$. We have the following positive lower bound on
        $\mu_n(\V_{\Sigma})$:
        \begin{equation}\label{mun(cSSigma}\mu_n(\V_\Sigma)\geq \prod_{p\notin \Sigma\cup \Omega}\left(1-\frac{\gamma_2}{p^2}\right)\times \left(\prod_{p\in \Omega\cup \Sigma_1}\frac{1}{p^{2d}}\right)\times \prod_{p\in \Sigma_2}\left(\frac{p^{(2d+2)}-2d p^{2d+1}}{p^{(2d+2)}-1}\right).\end{equation}
    \end{theorem}
\begin{proof}
    The proof is completely similar to that of Theorem~\ref{main thm of section 4} with the key input being Proposition~\ref{mu admissible proposition}.
\end{proof}

\section{Rational maps of degree $2$}\label{s 5}
In this section, we illustrate that some direct computations can be done when the degree is small  (for instance when the degree is $2$) in order to provide an explicit lower bound for the weak box densities of globally minimal rational maps.

For a homogeneous polynomial $R(X_0,\ldots,X_m)\in \Z[X_0,\ldots,X_m]$, a prime $p$, positive integers $1\leq i\leq n$, and a point $P\in\mathbb{P}^{m}(\Z/p^n\Z)$, the statement that 
$R(P)=0\bmod p^i$ is well-defined. This follows from the homogeneity of $R$ and the fact that any two liftings of $P$ in $\mathbb{A}^{m+1}(\Z/p^n\Z)$ differ (multiplicatively) by a unit of $\Z/p^n\Z$. We have:

\begin{lemma}\label{lem:A_p B_p S_p}
    Let $m$ be a positive integer, $R(X_0,\ldots,X_m)\in\Z[X_0,\ldots,X_m]$ a homogeneous polynomial of degree at least $2$. For each prime $p$, let
    \begin{align*}
        \bar{A}_p&=\{P\in(\Z/p\Z)^{m+1}:\ R(P)\neq 0\}\ \text{and}\\
        \bar{B}_p&=\left\{P\in(\Z/p\Z)^{m+1}:\ R(P)=0\ \text{and}\  \frac{\partial R}{\partial X_i}(P)\neq 0\ \text{for some $0\leq i\leq m$}\right\}.
    \end{align*}
    Then the set $\bar{\mathcal{S}}_p:=\{P\in\mathbb{P}^{m}(\Z/p^2\Z):\ R(P)\neq 0\bmod p^2\}$
    has
    $$|\bar{\mathcal{S}}_p|=\frac{p^{m+1}|\bar{A}_p|+(p^{m+1}-p^m)|\bar{B}_p|}{p(p-1)}.$$
\end{lemma}
\begin{proof}
    Each point of $\bar{A}_p$ lifts to $p^{m+1}$ many points in $(\Z/p^2\Z)^{m+1}$. 
    By Lemma~\ref{partial derivative lemma}, each point of $\bar{B}_p$ lifts to exactly $p^{m+1}-p^m$ many 
    points $P'\in (\Z/p^2\Z)^{m+1}$ such that $R(P')\neq 0\bmod p^2$. 
    
    Since $R$ is homogeneous of degree at least $2$, the set of all the above liftings is invariant under multiplication by elements of $(\Z/p^2\Z)^\times$ and every lifting has an entry that is in $(\Z/p^2\Z)^\times$. Hence these liftings yield exactly $(p^{m+1}|\bar{A}_p|+(p^{m+1}-p^m)|\bar{B}_p|)/(p(p-1))$ many points in $\mathbb{P}^m(\Z/p^2\Z)$. Finally, we note that if $P\in(\Z/p\Z)^{m+1}$ such that $R(P)=0$ and $\displaystyle\frac{\partial R}{\partial X_i}(P)=0$ for every $i$ then using Taylor expansion
    as in the proof of Lemma~\ref{partial derivative lemma} we have $R(P')=0\bmod p^2$
    for every lifting $P'$ of $P$ in $\mathbb{A}^{m+1}(\Z/p^2\Z)$.   
\end{proof}

Consider the generic polynomials
$$F(X,Y)=aX^2+bXY+cY^2\ \text{and}\ G(X,Y)=dX^2+eXY+fY^2.$$
Their resultant is
$$R=a^2f^2-abef-2acdf+ace^2+b^2df-bcde+c^2d^2\in\Z[a,b,c,d,e,f].$$
For each prime $p$, let $\bar{A}_p$, $\bar{B}_p$, and $\bar{\mathcal{S}}_p$ be as in Lemma~\ref{lem:A_p B_p S_p}. Let $\mathcal{S}_p\subseteq \cR^{(2)}$ be the preimage of 
$\bar{\mathcal{S}}_p$ under the reduction map
$\mathbb{P}^5(\Q)\rightarrow\mathbb{P}^5(\Z/p^2\Z)$. Then $\mathcal{S}=\bigcap_p \mathcal{S}_p$
is exactly the set of rational maps of degree $2$ defined over $\Q$ with squarefree resultant. This is also the set $\V_{\Sigma}$ in the previous section when $d=2$ and $\Sigma$ is the empty set.
Let $n$ be a natural number in the range $[1,6]$ and $\mu_n$ be the weak box density in $\mathbb{P}^5(\Q)$. Proposition~\ref{mu admissible proposition} asserts that  
\begin{equation}\label{mu_n(cS_p)=prod_p pi(cS_p)}
    \mu_n(\cS)=\prod_p \pi(\cS_p).
\end{equation}

From \eqref{eq:picSp definition} and Lemma~\ref{lem:A_p B_p S_p}, we have
\begin{equation}\label{eq:S and the A_p B_p}
 \pi(\mathcal{S})=\prod_p\pi(\mathcal{S}_p)=\prod_{p}\frac{p(p-1)|\bar{\mathcal{S}}_p|}{p^{12}(1-p^{-6})}=\prod_p\frac{p^6|\bar{A}_p|+(p^6-p^5)|\bar{B}_p|}{p^{12}(1-p^{-6})}.   
\end{equation}

We also let
\begin{align*}
    \bar{A}_p'&=\{P\in(\Z/p\Z)^6: R(P)=0\}\ \text{and}\\
    \bar{B}_p'&=\left\{P\in(\Z/p\Z)^6: R(P)=0\ \text{and}\ \frac{\partial R}{\partial a}(P)=\ldots=\frac{\partial R}{\partial f}(P)=0\right\}.
\end{align*}
We have
\begin{equation}\label{eq:relation among A_p, A_p',...}
    |\bar{A}_p|=p^{6}-|\bar{A}_p'|\ \text{and}\ |\bar{B}_p|=|\bar{A}_p'|-|\bar{B}_p'|.
\end{equation}
Therefore
\begin{equation}\label{eq:in terms of A' and B'}
    p^6|\bar{A}_p|+(p^6-p^5)|\bar{B}_p|=p^{12}-p^5|\bar{A}_p'|-(p^6-p^5)|\bar{B}_p'|.
\end{equation}

For small primes $p$ (for instance $p<20$), with the help of a home computer, we can easily check whether each $(a,b,c,d,e,f)\in (\Z/p\Z)^6$ belongs to $\bar{A}_p$ or $\bar{B}_p$ and obtain the following:

\begin{center}
\begin{tabular}{|c|c|c|c|}
\hline
$p$ & $|A_p|$ & $|B_p|$ & $\pi(\bar{\mathcal{S}}_p)$\\  \hline
& & & \\
$2$ & $24$& $18$ & $\displaystyle\frac{11}{21}$ \\ 
& & & \\
$3$ & $432$ & $192$ & $\displaystyle\frac{10}{13}$\\ 
& & & \\
$5$ & $12000$ & $2880$ & $\displaystyle\frac{596}{651}$\\ 
& & & \\
$7$ & $98784$ & $16128$ & $\displaystyle\frac{782}{817}$\\ 
& & & \\
$11$ & $1597200$ & $158400$ & $\displaystyle\frac{14510}{14763}$\\ 
& & & \\
$13$ & $4429152$ & $366912$ & $\displaystyle\frac{9460}{9577}$\\ 
& & & \\
$17$ & $22639104$ & $1410048$ & $\displaystyle\frac{11888}{11973}$\\ 
& & & \\
$19$ & $44446320$ & $2462400$ & $\displaystyle\frac{43314}{43561}$\\  
& & & \\ \hline
\end{tabular}
\end{center}

Then we have $\displaystyle\prod_{p<20}\pi(\mathcal{S}_p)=0.33843144838...$ We now aim to give a lower bound for $\pi(\mathcal{S}_p)$ when $p>20$. From \eqref{eq:S and the A_p B_p} and \eqref{eq:in terms of A' and B'}, it suffices to give upper bounds for $|\bar{A}_p'|$
and $|\bar{B}_p'|$.
\begin{proposition}\label{prop:barAp'}
     $|\bar{A}_p'|\leq  2p^5-p^3$.
\end{proposition}
\begin{proof}
    Write
    $$R=a^2f^2+a(ce^2-bef-2cdf)+b^2df-bcde+c^2d^2.$$
    Among the $(a,b,c,d,e,f)\in (\Z/p\Z)^6$ such that $R(a,\ldots,f)=0$, we divide into $3$ cases.

    \textbf{Case 1:} $f\neq 0$. There are $p^4(p-1)$ many such tuples $(b,c,d,e,f)$. 
    For each tuple, there are at most $2$ values of $a$ such that $R(a,\ldots,f)=0$. Hence we have at most
    $2p^4(p-1)$ many tuples $(a,\ldots,f)$ in this case.

    \textbf{Case 2:} $f=0$ and $ce\neq 0$. There are $p^2(p-1)^2$ many such tuples $(b,c,d,e,f)$. For each tuple, there is a unique $a$ such that $R(a,\ldots,f)=0$. Hence we have 
    $p^2(p-1)^2$ many tuples $(a,\ldots,f)$ in this case.

    \textbf{Case 3}: $f=0$ and $ce=0$. Then we must have $cd=0$ and $R(a,\ldots,f)=0$ for any $a$. This breaks into 2 smaller cases:
    \begin{itemize}
        \item $f=0$ and $c=0$. There are $p^4$ many tuples $(a,\ldots,f)$ in this case.
        \item $f=0$ and $c\neq 0$, hence $e=0$ and $d=0$. There are $p^2(p-1)$ many tuples
        $(a,\ldots,f)$ in this case.
    \end{itemize}
    Overall, we have:
    $$|\bar{A}_p'|\leq 2p^4(p-1)+p^2(p-1)^2+p^4+p^2(p-1)=2p^5-p^3.$$
\end{proof}

\begin{proposition}\label{prop:barBp'}
    When $p>3$, we have $|\bar{B}_p'|<p^4+3p^3$.
\end{proposition}
\begin{proof}
    Let $I$ be the ideal of $\Z[a,\ldots,f]$ generated by $R$ and its partial derivatives. By using Macaulay2\footnote{Available at \url{https://macaulay2.com/}}, we can compute the following primary decomposition in $\Q[a,\ldots,f]$:
    $$I\Q[a,\ldots,f]=(ce-bf,cd-af,bd-ae)\cap (e^2-4df,2cd-be+2af,b^2-4ac).$$

    Let $I_1=(ce-bf,cd-af,bd-ae)$ and $I_2=(e^2-4df,2cd-be+2af,b^2-4ac)$ as ideals of $\Z[a,\ldots,f]$. Using Macaulay2 again, we can check that:
    $$I_1\cap I_2\subseteq I\ \text{and}\ 12I\subseteq I_1\cap I_2.$$
    This means that inside the ring $Z[\frac{1}{12},a,\ldots,f]$, we have $I=I_1\cap I_2$. Since $p>3$, we have $\bar{B}_p'=\bar{C}_p\cup\bar{D}_p$ where
    $$\bar{C}_p:=\{(a,b,c,d,e,f)\in (\Z/p\Z)^6:\ ce-bf=cd-af=bd-ae=0\}\ \text{and}$$
    $$\bar{D}_p:=\{(a,b,c,d,e,f)\in (\Z/p\Z)^6:\ e^2-4df=2dc-be+2af=b^2-4ac=0\}.$$

    It is easy to determine $|\bar{C}_p|$: we simply determine the number of ordered pairs of linearly dependent vectors $((a,b,c),(d,e,f))$ in $(\Z/p\Z)^3$. When $(a,b,c)$ is zero, $(d,e,f)$ can be any vector. Otherwise, for any given non-zero $(a,b,c)$, there are exactly $p$ many choices for $(d,e,f)$. This gives
    \begin{equation}\label{eq:barCp}
    |\bar{C}_p|= p^3+p(p^3-1).
    \end{equation}
    We now give an upper bound for $|\bar{D}_p\setminus\bar{C}_p|$. 

    \textbf{Claim:} for $(a,b,c,d,e,f)\in \bar{D}_p\setminus\bar{C}_p$, we have $abcdef\neq 0$. We present the proof that $a\neq 0$, then the  proof for each of the statements $c\neq 0$, $d\neq 0$, and $f\neq 0$ is
    similar. After proving $acdf\neq 0$, since $e^2=4df$ and $b^2=4ac$ we have $be\neq 0$ as well.
    Suppose $a=0$, then from the definition of $\bar{D}_p$, we have:
    $$b=0,\ dc=0,\ e^2=4df$$
    Since $(a,b)=(0,0)$ we have $c\neq 0$, otherwise $(a,b,c,d,e,f)\in\bar{C}_p$. Hence $d=0$ and $e=0$. But then we still have $(a,b,c,d,e,f)\in\bar{C}_p$, contradiction.

    Now there are $(p-1)^3$ many tuples $(a,b,e)$ with $abe\neq 0$. We fix any such $(a,b,e)$. Then
    $c$ is determined uniquely from $b^2=4ac$. We now have the system of equations in $(d,f)$:
    $$df=e^2/4\ \text{and}\ 2dc-be+2af=0.$$
    By substituting $f=(be-2dc)/(2a)$ into the first equation, we get a quadratic equation in $d$. This explains why there are at most $2$ solutions $(d,f)$ of the above system.  Therefore
    \begin{equation}\label{eq:barDp}
    |\bar{D}_p\setminus\bar{C}_p|\leq 2(p-1)^3.
    \end{equation}

    From \eqref{eq:barCp} and \eqref{eq:barDp}, we have:
    $$|\bar{B}_p'|=|\bar{C}_p|+|\bar{D}_p\setminus\bar{C}_p|<p^4+3p^3.$$
    \end{proof}

    \begin{theorem}\label{d=2 main thm}
       We have 
        $$\mu_n(\cS)>\left(\prod_{p\geq 23}\frac{p^{12}-3p^{10}-2p^9+4p^8}{p^{12}(1-p^{-6})}\right)\cdot 0.33843144838...>0.327$$
    Therefore, with respect to the weak box densities, more than $32.7$\% rational maps of degree $2$ with rational coefficients have squarefree, hence minimal, resultant.
    \end{theorem}
\begin{proof}
    Recall that from \eqref{mu_n(cS_p)=prod_p pi(cS_p)}, we have that 
    \[\mu_n(\cS_p)=\prod_p \pi(\cS_p).\]
     From \eqref{eq:S and the A_p B_p}, \eqref{eq:in terms of A' and B'}, Proposition~\ref{prop:barAp'}, and Proposition~\ref{prop:barBp'}, the result follows.
\end{proof}
\bibliographystyle{alpha}
\bibliography{references}
\end{document}